\documentclass[master]{thesis}

\usepackage[cmex10]{amsmath}
\usepackage{amsthm,amssymb}

\usepackage[pdftex]{graphicx} 

\usepackage[caption=false]{subfig}

\usepackage{booktabs}

\usepackage{url}
\usepackage{cite} 

\newcommand{\cref}[1]{Chapter~\ref{#1}}  

\title{Multidimensional scaling: Infinite metric measure spaces}

\author{Lara Kassab}

\email{lara.kassab@colostate.edu}

\department{Department of Mathematics}

\semester{Spring 2019}

\advisor{Henry Adams}
\committee{Michael Kirby} 
\committee{Bailey Fosdick}

\mycopyright{
Copyright by Lara Kassab 2019\\
All Rights Reserved
}

\abstract{%
Multidimensional scaling (MDS) is a popular technique for mapping a finite metric space into a low-dimensional Euclidean space in a way that best preserves pairwise distances.
We study a notion of MDS on infinite metric measure spaces, along with its optimality properties and goodness of fit.
This allows us to study the MDS embeddings of the geodesic circle $S^1$ into $\mathbb{R}^m$ for all $m$, and to ask questions about the MDS embeddings of the geodesic $n$-spheres $S^n$ into $\mathbb{R}^m$.
Furthermore, we address questions on convergence of MDS.
For instance, if a sequence of metric measure spaces converges to a fixed metric measure space $X$, then in what sense do the MDS embeddings of these spaces converge to the MDS embedding of $X$?
Convergence is understood when each metric space in the sequence has the same finite number of points, or when each metric space has a finite number of points tending to infinity.
We are also interested in notions of convergence when each metric space in the sequence has an arbitrary (possibly infinite) number of points.
}

\acknowledgements{We would like to thank Henry Adams, Mark Blumstein, Bailey Fosdick, Michael Kirby, Henry Kvinge, Facundo M{\'e}moli, Louis Scharf, the students in Michael Kirby's Spring 2018 class, and the Pattern Analysis Laboratory at Colorado State University for their helpful conversations and support throughout this project.
}

\usepackage{amsmath,amssymb,amsthm}
\usepackage[hidelinks]{hyperref}
\usepackage{graphicx}
\usepackage{stackrel}
\usepackage{tikz} 
\usetikzlibrary{matrix,arrows}
\usepackage{color}
\usepackage{textcomp}
\usepackage{enumerate}
\usepackage{mathtools}
\newtheoremstyle{myremark} 
    {7pt}                    
    {7pt}                    
    {}  	                 
    {}                           
    {\bf}       	         
    {.}                          
    {.5em}                       
    {}  
    
\theoremstyle{plain}
\newtheorem{lemma}{Lemma}[section]

\newtheorem{corollary}[lemma]{Corollary}
\newtheorem{proposition}[lemma]{Proposition}
\newtheorem{theorem}[lemma]{Theorem}
\newtheorem*{theorem-main}{Theorem~\ref{thm:main}}
\newtheorem*{theorem-secondary}{Theorem~\ref{thm:secondary}}
\theoremstyle{definition}
\newtheorem{conjecture}[lemma]{Conjecture}
\newtheorem{definition}[lemma]{Definition}
\newtheorem{question}[lemma]{Question}
\theoremstyle{myremark}
\newtheorem{remark}[lemma]{Remark}
\newtheorem*{theorem*}{Theorem}

\newcommand{\bA}{\ensuremath{\mathbf{A}}}
\newcommand{\bB}{\ensuremath{\mathbf{B}}}
\newcommand{\bD}{\ensuremath{\mathbf{D}}}

\newcommand{\bGamma}{\ensuremath{\mathbf{\Gamma}}}
\newcommand{\bH}{\ensuremath{\mathbf{H}}}
\newcommand{\bI}{\ensuremath{\mathbf{I}}}

\newcommand{\bLambda}{\ensuremath{\mathbf{\Lambda}}}
\newcommand{\bX}{\ensuremath{\mathbf{X}}}
\newcommand{\C}{\ensuremath{\mathbb{C}}}

\newcommand{\N}{\ensuremath{\mathbb{N}}}

\newcommand{\R}{\ensuremath{\mathbb{R}}}

\newcommand{\cB}{\ensuremath{\mathcal{B}}}

\newcommand{\cG}{\ensuremath{\mathcal{G}}}
\newcommand{\cH}{\ensuremath{\mathcal{H}}}

\newcommand{\fD}{\ensuremath{\mathfrak{D}}}

\newcommand{\supp}{\ensuremath{\mathrm{supp}}}
\newcommand{\strain}{\ensuremath{\mathrm{Strain}}}
\newcommand{\stress}{\ensuremath{\mathrm{Stress}}}

\newcommand{\tr}{\ensuremath{\mathrm{Tr}}}
\newcommand*\diff{\mathop{}\!\mathrm{d}}

\begin{document} 

\frontmatter 

\maketitle    
\makemycopyright        
\makeabstract           
\makeacknowledgements   

\prelimtocentry{Dedication}
\begin{flatcenter}

    DEDICATION

    \vfill 

    \noindent \textit{I would like to dedicate this thesis to my parents.}
    \vfill 
\end{flatcenter}
\newpage

\tableofcontents   

\mainmatter

\chapter{Introduction}
\label{chap:intro}

Multidimensional scaling (MDS) is a set of statistical techniques concerned with the problem of constructing a configuration of $n$ points in a Euclidean space using information about the dissimilarities between the $n$ objects.
The dissimilarities need not be based on Euclidean distances; they can represent many types of dissimilarities between objects. 
The goal of MDS is to map the objects $x_1,  \ldots, x_n$ to a configuration (or embedding) of points $f(x_1), \ldots, f(x_n)$ in $\R^m$ in such a way that the given dissimilarities $d_{ij}$ are well-approximated by the Euclidean distances between $f(x_i) $ and $f(x_j)$.
The different notions of approximation give rise to the different types of MDS, and the choice of the embedding dimension $m$ is arbitrary in principle, but low in practice ($m = 1, 2, \mbox{or } 3$).

MDS is an established multivariate analysis technique used in a multitude of disciplines. 
It mainly serves as a visualization technique for proximity data, the input of MDS, which is usually represented in the form of an $n \times n$ dissimilarity matrix. 
Proximity refers to similarity and dissimilarity measures; these measures are essential to solve many pattern recognition problems such as classification and clustering.
A frequent source of dissimilarities is distances between high-dimensional objects, and in this case, MDS acts as an (often nonlinear) dimension reduction technique.

MDS is indeed an optimization problem because a perfect Euclidean embedding preserving the dissimilarity measures does not always exist.
If the dissimilarity matrix can be realized exactly as the distance matrix of some set of points in $\R^m$ (i.e.\ if the dissimilarity matrix is \emph{Euclidean}), then MDS will find such a realization.
Furthermore, MDS can be used to identify the minimum such Euclidean dimension $m$ admitting an isometric embedding.
However, some dissimilarity matrices or metric spaces are inherently non-Euclidean (cannot be embedded into $\R^m$ for any $m$).
When a dissimilarity matrix is not Euclidean, then MDS produces a mapping into $\R^m$  that distorts the interpoint pairwise distances as little as possible, in a sense that can be made precise.

The various types of MDS arise mostly from the different loss functions they minimize, and they mainly fall into two categories: metric and non-metric MDS. 
A brief overview on the different types of MDS is given in Section~\ref{sec: mMDS}.
One of the main methods of MDS is commonly known as classical multidimensional scaling (cMDS), which minimizes a form of a loss function known as $\strain$. The classical MDS algorithm is algebraic and not iterative. 
Therefore, it is simple to implement, and is guaranteed to discover the optimal configuration in $\R^m$.
In Section~\ref{sec: cMDS}, we describe the algorithm, and discuss its optimality properties and goodness of fit. 

A \emph{metric measure space} is a triple $(X,d_X,\mu_ X)$ where $(X,d_X)$ is a compact metric space, and $\mu_X$ is a Borel probability measure on $X$. 
In this work, we study a notion of MDS on infinite metric measure spaces, which can be simply thought of as spaces of infinitely many points equipped with some probability measure.
Our motivation is to prove convergence properties of MDS of metric measure spaces. 
That is, if a sequence of metric measure spaces $X_n$ converges to a fixed metric measure space $X$ as $n \to \infty$, then in what sense do the MDS embeddings of these spaces converge to the MDS embedding of $X$?
Convergence is well-understood when each metric space has the same finite number of points, and also fairly well-understood when each metric space has a finite number of points tending to infinity. 
An important example is the behavior of MDS as one samples more and more points from a dataset.
We are also interested in convergence when the metric measure spaces in the sequence perhaps have an infinite number of points.
In order to prove such results, we first need to define the MDS embedding of an infinite metric measure space $X$, and study its optimal properties and goodness of fit.

In Section~\ref{ss:infinite-mds}, we explain how MDS generalizes to possibly infinite metric measure spaces.
We describe an infinite analogue to the classical MDS algorithm.
Furthermore, in Theorem~\ref{Thm: infinite-mds-optimization} we show that this analogue minimizes a $\strain$ function similar to the $\strain$ function of classical MDS. 
This theorem generalizes~\cite[Theorem~14.4.2]{bibby1979multivariate}, or equivalently~\cite[Theorem~2]{trosset1997computing}, to the infinite case.
Our proof is organized analogously to the argument in~\cite[Theorem~2]{trosset1997computing}.

As a motivating example, we consider the MDS embeddings of the circle equipped with the (non-Euclidean) geodesic metric.
By using the properties of circulant matrices, we carefully identify the MDS embeddings of evenly-spaced points from the geodesic circle into $\R^m$, for all $m$.
As the number of points tends to infinity, these embeddings lie along the curve
\[\sqrt{2}\left(\cos\theta, \sin\theta, \tfrac{1}{3}\cos3\theta, \tfrac{1}{3}\sin3\theta,
\tfrac{1}{5}\cos5\theta, \tfrac{1}{5}\sin5\theta,\ldots\right)\in\R^m.\]

Furthermore, we address convergence questions for MDS.
Indeed, convergence is well-understood when each metric space has the same finite number of points~\cite{sibson1979studies}, but we are also interested in convergence when the number of points varies and is possibly infinite.
We survey Sibson's perturbation analysis~\cite{sibson1979studies} for MDS on a fixed number of $n$ points. 
We survey results of~\cite{bengio2004learning,koltchinskii2000random} on the convergence of MDS when $n$ points $\{x_1,\ldots,x_n\}$ are sampled from a metric space according to a probability measure $\mu$, in the limit as $n\to\infty$.
We reprove these results under the (simpler) deterministic setting when points are not randomly chosen, and instead we assume that the corresponding finite measures $\mu_n = \frac{1}{n}\sum\limits_{i=1}^{n} \delta_{x_i}$ (determined by $n$ points) converge to $\mu$.
This allows us, in Section~\ref{sec: convergence deterministic arbitrary}, to consider the more general setting where we have convergence of \emph{arbitrary} probability measures $\mu_n\to\mu$, where now each measure $\mu_n$ is allowed to have infinite support.

In Chapter~\ref{chap: related work}, we survey related work.
In Chapter~\ref{chap: preliminaries}, we present background information on proximities, metric spaces, spaces with various structures such as inner products or norms, and metric measure spaces. 
We present an overview on the theory of MDS in Chapter~\ref{chap: MDS theory}. 
We briefly describe the different types of MDS with most emphasis on classical MDS. 
In Chapter~\ref{chap: Operator Theory}, we present necessary background information on operator theory and infinite-dimensional linear algebra. 
We define a notion of MDS for infinite metric measure spaces in Chapter~\ref{chap: iMDS}.
In Chapter~\ref{chap: MDS circle}, we identify the MDS embeddings of the geodesic circle into $\R^m$, for all $m$, as a motivating example.
Lastly, in Chapter~\ref{chap: convergence}, we describe different notions of convergence of MDS.

\chapter{Related Work}
\label{chap: related work}

The reader is referred to the introduction of~\cite{trosset1998new} and to~\cite{de198213, groenen2014past} for some aspects of history of MDS.
Furthermore, the reader is referred to~\cite{smithies1970} for the theory of linear equations, mainly of the second kind, associated with the names of Volterra, Fredholm, Hilbert and Schmidt. 
The treatment has been modernised by the systematic use of the Lebesgue integral, which considerably widens the range of applicability of the theory. 
Among other things, this book considers singular functions and singular values as well.
 
There are a variety of papers that study some notion of robustness or convergence of MDS.
In a series of papers~\cite{sibson1978studies,sibson1979studies,sibson1981studies}, Sibson and his collaborators consider the robustness of multidimensional scaling with respect to perturbations of the underlying distance or dissimilarity matrix.
Indeed, convergence is well-understood when each metric space has the same finite number of points~\cite{sibson1979studies}, but we are also interested in convergence when the number of points varies and is possibly infinite.
The paper~\cite{bengio2004learning} studies the convergence of MDS when more and more points are sampled independent and identically distributed (i.i.d.) from an unknown probability measure $\mu$ on $X$.
The paper~\cite{koltchinskii2000random} presents a key result on convergence of eigenvalues of operators.
Furthermore,~\cite[Section 3.3]{pekalska2001generalized} considers embedding new points in psuedo-Euclidean spaces,~\cite[Section 3]{diaconis2008horseshoes} considers infinite MDS in the case where the underlying space is an interval (equipped with some metric), and~\cite[Section 6.3]{buja2008data} discusses MDS on large numbers of objects.

Some popular non-linear dimensionality reduction techniques besides MDS include  Isomap~\cite{tenenbaum2000global}, 
Laplacian eigenmaps~\cite{belkin2003laplacian}, Locally Linear Embedding (LLE)~\cite{roweis2000nonlinear}, and Nonlinear PCA~\cite{scholz2005non}. 
See also the recent paper~\cite{kvinge2018gpu} for a GPU-oriented dimensionality reduction algorithm that is inspired by Whitney's embedding theorem.
The paper~\cite{williams2001connection} makes a connection between kernel PCA and metric MDS, remarking that kernel PCA is a form of MDS when the kernel is isotropic.
The reader is referred to~\cite{bengio2004learning} for further relations between spectral embedding methods and kernel PCA.

\section{Applications of MDS} 

MDS is an established multivariate analysis technique used in a multitude of disciplines like social sciences, behavioral sciences, political sciences, marketing, etc. 
One of the main advantages of MDS is that we can analyze any kind of proximity data, i.e.\ dissimilarity or similarity measures. 
For instance, MDS acts as an (often nonlinear) dimension reduction technique when the dissimilarities are distances between high-dimensional objects.
Furthermore, when the dissimilarities are shortest-path distances in a graph, MDS acts as a graph layout technique~\cite{buja2008data}.
Furthermore, MDS is used in machine learning in solving classification problems. 
Some related developments in machine learning include Isomap~\cite{tenenbaum2000global} and 
kernel PCA~\cite{scholkopf1998nonlinear}. The reader is referred to~\cite{borg2005modern, buja2008data, cox2000multidimensional} for further descriptions on various applications of MDS.

\section{Multivariate Methods Related to MDS}

There exists several multivariate methods related to MDS. Some include Principal Component Analysis (PCA), Correspondence Analysis, Cluster Analysis and Factor Analysis~\cite{borg2005modern}. 
For instance, PCA finds a low-dimensional embedding of the data points that best preserves their variance as measured in the high-dimensional input space.
Classical MDS finds an embedding that preserves the interpoint distances, and it is equivalent to PCA when those distances are Euclidean.
The reader is referred to~\cite[Chapter 24]{borg2005modern} for a detailed comparison between the first three methods and MDS.
In Factor Analysis, the similarities between objects are expressed in the correlation matrix.
With MDS, one can analyze any kind of similarity or dissimilarity matrix, in addition to correlation matrices.

\chapter{Preliminaries}
\label{chap: preliminaries}

We describe preliminary material on proximities and metrics, on inner product and normed spaces, on metric measure spaces, and on distances between metric measure spaces.

\section{Proximities and Metrics}\label{sec: proximities}
\emph{Proximity} means nearness between two objects in a certain space.
It refers to similarity and dissimilarity measures, which are essential to solve many pattern recognition problems such as classification and clustering. 
Two frequent sources of dissimilarities are high-dimensional data and graphs~\cite{buja2008data}. 
Proximity data between $n$ objects is usually represented in the form of an $n \times n$ matrix. 
When the proximity measure is a \emph{similarity}, it is a measure of how similar two objects are, and when it is a \emph{dissimilarity}, it is a measure of how dissimilar two objects are. 
See~\cite[Section~1.3]{cox2000multidimensional} for a list of some of the commonly used similarity and dissimilarity measures.  
MDS is a popular technique used in order to visualize proximities between a collection of objects. 
In this section, we introduce some of the terminology and define some of the proximity measures discussed in the following chapters.

\begin{definition}
An $(n \times n)$ matrix $\mathbf D$ is called a \emph{dissimilarity matrix} if it is
symmetric and
\[d_{rr} = 0, \quad\text{with}\quad d_{rs} \geq 0 \quad\text{for}\quad r \neq s. \]
\end{definition}

The first property above is called refectivity, and the second property is called nonnegativity.
Note that there is no need to satisfy the triangle inequality.

\begin{definition}
A function $d\colon X \times X \to \R$ is called a \emph{metric} if the following conditions are fulfilled for all $ x,y,z \in X$:
\begin{itemize}
\item (reflectivity) $d(x, x) = 0$;
\item (positivity) $d(x, y) > 0$ for $x \neq y$;
\item (symmetry) $d(x, y) = d(y, x)$;
\item (triangle inequality) $d(x, y) \le d(x, z) + d(z, y)$.
\end{itemize}
\end{definition}

A \emph{metric space} $(X, d)$ is a set $X$ equipped with a metric $d \colon X \times X \to \R$.

\begin{definition}
The \emph{Euclidean distance} between two points $p = (p_1, \ldots, p_n)$ and $ q = (q_1, \ldots, q_n)$ in $\R^n$ is given by the formula,

\[ d( p, q) = \sqrt{\sum\limits_{i=1}^n (p_i - q_i)^2}. \]

\end{definition}

\begin{definition}
A dissimilarity matrix $\mathbf D$ is called \emph{Euclidean} if there exists a
configuration of points in some Euclidean space whose interpoint distances are given by $\mathbf D$; that is, if for some $m$, there exists points $x_1, \ldots, x_n \in \R^m$ such that 
\[d^2_{rs} = (x_r - x_s)^{\top}(x_r - x_s), \]
where $\top$ denotes the transpose of a matrix.
\end{definition}

\begin{definition}
A map $\phi : X \rightarrow Y$ between metric spaces $(X, d_X)$ and $(Y, d_Y)$ is an \emph{isometric embedding} if $d_Y(\phi (x), \phi (x')) = d_X(x, x')$ for all $x, x' \in X$.
The map $\phi$ is an \emph{isometry} if it is a surjective isometric embedding.
\end{definition}

Proximities in the form of similarity matrices also commonly appear in applications of MDS.
A reasonable measure of \emph{similarity} $s(a, b)$ must satisfy the following properties~\cite{bibby1979multivariate}:
\begin{itemize}
\item (symmetry) $s(a,b) = s(b,a)$,
\item (positivity) $s(a,b) >0$,
\item $s(a,b)$ increases as the similarity between $a$ and $b$ increases.
\end{itemize}

\begin{definition}
An $(n \times n)$ matrix $\mathbf C$ is called a \emph{similarity matrix} if it is symmetric and
\[c_{rs} \leq c_{rr} \quad \mbox{ for all } r,s. \]
\end{definition}

To use the techniques discussed in this thesis, it is necessary to transform the similarities to dissimilarities.
The standard transformation from a similarity matrix $\mathbf C$ to a dissimilarity matrix $\bD$ is defined by
\[ d_{rs} = \sqrt{c_{rr} -2 c_{rs} + c_{ss}}.\]

\section{Inner Product and Normed Spaces}\label{sec: spaces}

\begin{definition}
An \emph{inner product} on a linear space $X$ over the field $F$ $(\R \mbox{ or } \C)$ is a function $\langle \cdot ,\cdot \rangle \colon X\times X\to F$
with the following properties:

\begin{itemize}
\item (conjugate symmetry) $\langle x,y \rangle =\overline {\langle y,x\rangle}$ for all $x,y \in X$, where the overline denotes the complex conjugate;

\item (linearity in the first argument) \begin{align*}\langle \alpha x,y\rangle &=\alpha\langle x,y\rangle \\\langle x+y,z\rangle &=\langle x,z\rangle +\langle y,z\rangle\end{align*}
for all $x, y, z \in X$ and $\alpha \in F$;

\item (positive-definiteness)
\begin{align*}\langle x,x\rangle &\geq 0\\\langle x,x\rangle &=0\Leftrightarrow x=\mathbf {0} \end{align*}
for all $x \in X$.

\end{itemize}

\end{definition}

An \emph{inner product space} is a vector space equipped with an inner product.

\begin{definition}
A \emph{norm} on a linear space $X$ is a function $\| \cdot \| \colon X \to \R$ with the following properties:

\begin{itemize}
\item (nonnegative) $\|x\| \geq 0$, for all $x \in X$;
\item (homogeneous) $\| \lambda x \| = | \lambda| \|x \|$, for all $x \in X$ and $\lambda \in \R$ (or $\C$);
\item (triangle inequality) $\|x+y\| \leq \|x \| + \|y \|$ for all $x  , y \in X$;
\item (strictly positive) $\| x \| = 0$ implies $x=0$.

\end{itemize}
\end{definition}

A \emph{normed linear space} $(X, \| \cdot \|)$ is a linear space $X$ equipped with a norm $\| \cdot \|$.

\begin{definition}
A \emph{Hilbert Space} is a complete inner product space with norm defined by the inner product,
\[ \|f\|=\sqrt{\langle f,f \rangle}.\]
\end{definition}

\begin{definition} The space $\ell^2$ is a subspace of $\R^{\N}$ consisting of all sequences $(x_n)_{n=1}^{\infty}$ such that \[ \sum\limits _{n}|x_{n}|^2 <\infty. \]
\end{definition}

Define an inner product on $\ell ^2$ for all sequences $(x_n),(y_n) \in \ell^2$ by \[ \langle (x_n), (y_n) \rangle = \sum _{n=1} ^ \infty x_n y_n. \] 

The real-valued function $\|\cdot \|_{2} \colon \ell^2 \to \R$ defined by $\|(x_n)\|_{2}=\left(\sum\limits \limits_{n=1}^\infty|x_{n}|^{2}\right)^{1/2}$ defines a norm on $\ell ^2$.

\section{Metric Measure Spaces}\label{sec: metric measure spaces}
The following introduction to metric measure spaces and the metrics on them is based on~\cite{memoli2011gromov}. 
The reader is referred to~\cite{memoli2011gromov, memoli2014gromov} for detailed descriptions and interpretations of the following concepts in the context of object matching.

\begin{definition} The support of a measure $\mu$ on a metric space $(Z,d)$, denoted by $\supp[\mu]$, is the minimal closed subset $Z_0 \subseteq  Z$ such that
$\mu (Z \backslash Z_0) = 0$.
\end{definition}

Given a metric space $(X, d_X)$, by a measure on $X$ we mean a measure on $(X,\cB(X))$, where $\cB(X)$ is the Borel $\sigma$-algebra of $X$.
Given measurable spaces $(X,\cB(X))$ and $(Y,\cB(Y ))$ with measures $\mu_X$ and $\mu_Y$, respectively, let $\cB(X\times Y)$ be the $\sigma$--algebra on $X\times Y$ generated by subsets of the form $A \times B$ with $A \in \cB(X)$ and $B \in \cB(Y)$.
The product measure $\mu_X \otimes \mu_Y$ is defined to be the unique measure on $(X \times Y, \cB(X \times Y))$ such that $\mu_X \otimes \mu_Y (A \times B) = \mu_X(A) \mu_Y (B)$
for all $A \in \cB(X)$ and $B \in \cB(Y)$. Furthermore, for $x \in X$, let $\delta_x^X$ denote the Dirac measure on $X$.

\begin{definition}
For a measurable map $f : X\to Y$ between two compact metric spaces $X$ and $Y$, and for $\mu$ a measure on $X$, the \emph{push-forward measure $f_\# \mu$} on $Y$ is given by $f_\# \mu(A) =
\mu(f^{-1}(A))$ for $A \in \cB(Y)$.
\end{definition}

\begin{definition} A \emph{metric measure space} is a triple $(X,d_X,\mu_ X)$ where
\begin{itemize}
\item $(X,d_X)$ is a compact metric space, and
\item $\mu_X$ is a Borel probability measure on $X$, i.e.\ $\mu_X (X) = 1$.
\end{itemize}
\end{definition}

In the definition of a metric measure space, it is sometimes assumed that $\mu_X$ has full support, namely that $\supp[\mu_X] = X$. 
When this is done, it is often for notational convenience.
For the metric measure spaces that appear in Section~\ref{ss:converge-metric} we will assume that $\supp[\mu_X] = X$, but we will not make this assumption elsewhere in the document.

Denote by $\cG_w$ the collection of all metric measure spaces.
Two metric measure spaces $(X, d_X, \mu_X)$ and $(Y, d_Y, \mu_Y)$ are called \emph{isomorphic} if and only if there exists an isometry $\psi \colon X \to Y$ such that
$(\psi)_\#\mu_X = \mu_Y$.

\section{Metrics on Metric Measure Spaces}\label{ss:converge-metric}

We describe a variety of notions of distance between two metric spaces or between two metric measure spaces.
The content in this section will only be used in Section~\ref{sec: convergence GW}. We begin by first defining the Hausdorff distance between two aligned metric spaces, and the Gromov--Hausdorff distance between two unaligned metric spaces. The Gromov--Wasserstein distance is an extension of the Gromov--Hausdorff distance to metric measure spaces.
In this section alone, we assume that the metric measure spaces $(X,d_X,\mu_ X)$ that appear satisfy the additional property that $\supp[\mu_X] = X$.

\begin{definition} Let $(Z, d)$ be a compact metric space.
The \emph{Hausdorff distance} between any two closed sets $A$, $B \subseteq Z$ is defined as 
\[ d_\mathcal{H}^Z(A,B) = \max\{\,\sup _{{a\in A}}\inf _{{b\in B}}d(a,b),\,\sup _{{b\in B}}\inf _{{a\in A}}d(a,b)\,\}{\mbox{.}}\!\]
\end{definition}

\begin{definition} The \emph{Gromov--Hausdorff distance} between compact metric spaces X and Y is defined as \begin{equation}d_{\mathcal{GH}}(X,Y) = \inf _{{Z, f, g}}d_\mathcal{H}^Z(f(X),g(Y))\end{equation}
where $f : X \rightarrow Z$ and $g : Y \rightarrow Z$ are isometric embeddings into the metric space $(Z, d)$.
\end{definition}

It is extremely difficult to compute Gromov--Hausdorff distances; indeed the set of all metric spaces $Z$ admitting isometric embeddings from both $X$ and $Y$ is an unbelievably large collection over which to take an infimum.
It is possible to give an equivalent definition of Gromov--Hausdorff distances instead using correspondences (which feel more computable though in practice are still hard to compute.) 

\begin{definition}
A subset $R \subseteq X \times Y$ is said to be a \emph{correspondence} between sets $X$ and $Y$ whenever $\pi _1(R) = X$
and $\pi _2(R) = Y$, where $\pi _1 \colon X \times Y \rightarrow X$ and $\pi_2 \colon X \times Y \rightarrow Y$ are the canonical projections.
\end{definition}
Let $\mathcal{R}(X,Y)$ denote the set of all possible correspondences between sets $X$ and $Y$.
For metric spaces $(X, d_X)$ and $(Y, d_Y)$, let the \emph{distortion} $\Gamma _{X,Y} \colon X \times Y \times X \times Y \to \R^{+}$ be given by
\[\Gamma _{X,Y} (x,y,x',y') = | d_X(x, x') - d_Y (y, y')|.\]

\begin{definition} (Alternative form of $ d_{\mathcal{GH}}$) One can equivalently (in the sense of equality)
define the \emph{Gromov--Hausdorff distance} between compact metric spaces $(X, d_X)$ and $(Y, d_Y)$ as 
\begin{equation} 
d_{\mathcal{GH}}(X,Y) = \frac{1}{2} \inf_{{R \in \mathcal {R}(X,Y)}} \sup_{\substack{(x,y)\in R \\ (x',y') \in R}} \Gamma _{X,Y} (x,y,x',y'),
\end{equation}
where $R$ ranges over $\mathcal{R}(X,Y)$.
\end{definition}

We now build up the machinery to describe a notion of distance not between metric spaces, but instead between metric measure spaces.
Let $\mathcal{C}(Z)$ denote the collection of all compact subsets of $Z$.
We denote the collection of all \emph{weighted objects} in the metric space $(Z,d)$ by
\[\mathcal{C}_w(Z) := \{(A, \mu _A)~|~A \in \mathcal{C}(Z)\},
\]
where for each $A \in \mathcal{C}(Z)$, $\mu _A$ is a Borel probability measure with $\supp[\mu _A] = A$.
Informally speaking, an object in $\mathcal{C}_w(Z)$ is specified not only by the set of points that constitute it, but also by a distribution of importance over these points.
These probability measures can be thought of as acting as weights for each point in the metric space~\cite{memoli2011gromov}.

This following relaxed notion of correspondence between objects is called a matching measure, or a coupling.

\begin{definition}\label{def: matching measure} (Matching measure)
Let $(A,\mu_A),(B,\mu_B) \in C_w(Z)$.
A measure $\mu$ on the product space $A \times B$ is a \emph{matching measure} or \emph{coupling} of $\mu _A$ and $\mu_B$ if
\[\mu (A_0 \times B) = \mu_ A(A_0) \mbox{ \quad and \quad } \mu (A \times B_0) = \mu_B(B_0)\]
for all Borel sets $A_0 \subseteq A$ and $B_0 \subseteq B$.
Denote by $\mathcal{M}(\mu_A,\mu_B)$ the set of all couplings of $\mu_A$ and $\mu_B$.
\end{definition}

\begin{proposition}\cite[Lemma 2.2]{memoli2011gromov}\label{prop: matching measure}
Let $\mu_A$ and $\mu _B$ be Borel probability measures on $(Z, d)$, a compact
space, with $\supp(\mu_A)=\supp(\mu_B)=Z$. If $ \mu \in \mathcal M (\mu_A, \mu_B)$, then $\mathcal R(\mu) := \supp [\mu]$ belongs to $\mathcal R(\supp[\mu_A], \supp[\mu_B])$.
\end{proposition}

\begin{definition}[Wasserstein--Kantorovich--Rubinstein distances
between measures]
For each $p \geq 1$, the
following family of distances on $\mathcal C_w(Z)$ known as the \emph{Wasserstein distances}, where $(Z, d)$ is a compact metric space:
\[ d^Z_{W,p}(A,B) = \left(\inf _{\mu \in \mathcal M(\mu_A ,\mu_B )}\int _{A\times B}d(a,b)^p \diff \mu (a,b) \right)^{1/p}, \]
for $1 \leq p < \infty$, and
\[ d^Z_{W,\infty}(A,B) = \inf _{\mu \in \mathcal M(\mu_A ,\mu_B )} \sup _{(a,b) \in \mathcal R(\mu)} d(a,b).\]
\end{definition}
In~\cite{memoli2011gromov}, M\'emoli introduced and studied a metric $\fD_p$ on $\cG_w$, defined below.
An alternative metric $\cG_p$ on $\cG_w$ is defined and studied by Strum in~\cite{sturm2006geometry}.
The two distances are not equal.
Theorem~5.1 of~\cite{memoli2011gromov} proves that $\cG_p \geq \fD_p$ for $ 1 \leq p \leq \infty$ and that $\cG_\infty = \fD_\infty$, where for $p < \infty$ the equality does not hold in general.

For $p \in [1, \infty )$ and $\mu \in \mathcal{M}(\mu_ X, \mu_Y)$, let
\[\mathbf{J}_p(\mu) = \frac{1}{2} \left( \int_{X \times Y} \int_{X \times Y} (\Gamma _{X,Y} (x,y,x',y'))^p \mu(\diff x \times dy) \mu (\diff x' \times \diff y')\right)^{\frac{1}{p}},\]
and also let 
\[\mathbf{J}_\infty (\mu)= \frac{1}{2} \sup_{\substack{x,x' \in X \\ y, y' \in Y \\ (x,y),(x',y') \in \mathcal R (\mu) }}\Gamma _{X, Y}(x,y,x',y').\]

\begin{definition}
For $1\leq p \leq \infty$, define the \emph{Gromov--Wasserstein distance} $\fD_p$ between two metric measure spaces $X$ and $Y$ as
\begin{equation*}
\fD_p(X, Y ) = \inf_{{\mu \in \mathcal{M}(\mu _X,\mu _Y)}} \mathbf{J}_p(\mu).
\end{equation*}
\end{definition}

In Section~\ref{sec: convergence GW}, we pose some questions relating the Gromov--Wasserstein distance to notions of convergence of MDS for possibly infinite metric measure spaces.

\chapter{The Theory of Multidimensional Scaling}
\label{chap: MDS theory}

Multidimensional scaling (MDS) is a set of statistical techniques concerned with the problem of constructing a configuration of $n$ points in a Euclidean space using information about the dissimilarities between the $n$ objects. 
The dissimilarities between objects need not be based on Euclidean distances; they can represent many types of dissimilarities. 
The goal of MDS is to map the objects $x_1,  \ldots, x_n$ to a configuration (or embedding) of points $f(x_1), \ldots, f(x_n)$ in $\R^m$ in such a way that the given dissimilarities $d(x_i,x_j)$ are well-approximated by the Euclidean distance $\|f(x_i) - f(x_j)\|_2$.
The different notions of approximation give rise to the different types of MDS.

If the dissimilarity matrix can be realized exactly as the distance matrix of some set of points in $\R^m$ (i.e.\ if the dissimilarity matrix is \emph{Euclidean}), then MDS will find such a realization.
Furthermore, MDS can be used to identify the minimum such Euclidean dimension $m$ admitting an isometric embedding.
However, some dissimilarity matrices or metric spaces are inherently non-Euclidean (cannot be embedded into $\R^m$ for any $m$).
When a dissimilarity matrix is not Euclidean, then MDS produces a mapping into $\R^m$  that distorts the interpoint pairwise distances as little as possible.
Though we introduce MDS below, the reader is also referred to~\cite{bibby1979multivariate, cox2000multidimensional, groenen2014past} for more complete introductions to MDS.

\section{Types of Multidimensional Scaling}  
There are several types of MDS, and they differ mostly in the loss function they minimize.
In general, there are two dichotomies (the following discussion is from~\cite{buja2008data}) :
\begin{enumerate}
\item Kruskal--Shepard distance scaling versus classical Torgerson--Gower inner-product scaling: In distance scaling, dissimilarities are fitted by distances, whereas classical scaling transforms the dissimilarities to a form that is naturally fitted by inner products.
\item Metric scaling versus nonmetric scaling: Metric scaling uses the actual values of the dissimilarities, while nonmetric scaling effectively uses only their ranks, i.e., their orderings~\cite{kruskal1964multidimensional, shepard1962analysis1, shepard1962analysis2}.
Nonmetric MDS is realized by estimating an optimal monotone transformation $f(d_{ij})$ of the dissimilarities while simultaneously estimating the configuration.
\end{enumerate}

There are two main differences between classical and distance scaling. First, inner products rely on an origin, while distances do not. So, a set of inner products determines uniquely a set of distances, but a set of distances determines a set of inner products only modulo change of origin. To avoid arbitrariness, one constrains classical scaling to mean-centered configurations. Second, distance scaling requires iterative minimization while classical scaling can be solved in a single step by computing an inexpensive eigendecomposition. The different types of MDS arise from different combinations of \{metric,
nonmetric\} with \{distance, classical\}. The reader is referred to~\cite{buja2008data} for further discussion on the topic.

Two common loss functions are known as the $\strain$ and $\stress$ functions. We call ``\strain'' any loss function that measures the lack of fit between inner products $\langle f(x_i), f(x_j) \rangle$ of the configuration points in $\R^m$ and the inner-product data $b_{ij}$ of the given data points. The following is an example of a $\strain$ function, where $f$ is the MDS embedding map:
\[\strain(f)=\sum\limits _{i,j}{(b_{ij}-\langle f(x_{i}),f(x_{j})\rangle )}^{2}.\]

We call ``\stress'' any loss function that measures the lack of fit between the Euclidean distances $\hat d_{ij}$ of the configuration points and the given proximities $\delta_{ij}$.
The general form of $\stress$~\cite{cox2000multidimensional, timm2012investigation} is 
\[\stress(f)= \sqrt{\frac {\sum\limits _{i,j}{(h(\delta_{ij})-\hat d_{ij})}^{2}}{scale}},\]
where $f$ is the MDS embedding map, where $h$ is a smoothing function of the data, and where the `$scale$' component refers to a constant scaling factor, used to keep the value of $\stress$ in the convenient range between 0 and 1.
The choice of $h$ depends on the type of MDS needed.
In metric scaling, $h$ is the identity map, which means that the raw input proximity data is compared directly to the mapped distances.
In non-metric scaling, however, $h$ is usually an arbitrary monotone function that can be optimized over.
The reader is referred to~\cite{borg2005modern, cox2000multidimensional, timm2012investigation} for descriptions of the most common forms of $\stress$.

\section{Metric Multidimensional Scaling}\label{sec: mMDS}
Suppose we are given $n$ objects $x_1, \ldots, x_n$ with the dissimilarities $d_{ij}$ between them, for $i,j = 1, \ldots, n$. Metric MDS attempts to find a set of points $f(x_1), \ldots, f(x_n)$ in a Euclidean space of some dimension where each point represents one of the objects, and the distances between points $\hat d_{ij}$ are such that
\[\hat d_{ij} \approx h(d_{ij}).\]
Here $h$ is typically the identity function, but could also be a continuous parametric monotone function that attempts to transform the dissimilarities to a distance-like form~\cite{cox2000multidimensional}. Assuming that all proximities are already in a
satisfactory distance-like form, the aim is to find a mapping $f$, for which $d_{ij}$ is approximately equal to $\hat d_{ij}$, for all $i,j$.
The two main metric MDS methods are classical scaling and least squares scaling.
We will introduce both, with most emphasis placed on the former. 

\subsection{Classical Scaling}\label{sec: cMDS} Classical multidimensional scaling (cMDS) is also known as Principal Coordinates Analysis (PCoA), Torgerson Scaling, or Torgerson--Gower scaling. 
The cMDS algorithm minimizes a $\strain$ function, and one of the main advantages of cMDS is that its algorithm is algebraic and not iterative.
Therefore, it is simple to implement, and is guaranteed to discover the optimal configuration in $\R^m$.
In this section, we describe the algorithm of cMDS, and then we discuss its optimality properties and goodness of fit. 

Let $\bD = (d_{ij})$ be an $n\times n$ dissimilarity matrix.
Let $\bA = (a_{ij})$, where $a_{ij} = -\frac{1}{2}d^2_{ij}$.
Define
\begin{equation}\label{B definition}
\bB = \bH\bA\bH,
\end{equation}
where $\bH = \bI - n^{-1}\mathbf{1}\mathbf{1}^{\top}$ is the centering matrix of size $n$ ($\mathbf{1}$ is a column-vector of $n$ ones and ${\top}$ denotes matrix transpose).
Multiplying $\bA$ by the matrix $\bH$ on either side has the effect of double-centering the matrix.
Indeed, we have
\[ b_{rs}=a_{rs}-\overline{a_{r\cdot}}-\overline{a_{\cdot s}}+\overline{a_{\cdot\cdot}},\]
where $\overline{a_{r\cdot}}=\frac{1}{n}\sum\limits_{s=1}^n a_{rs}$ is the average of row $r$, where $\overline{a_{\cdot s}}=\frac{1}{n}\sum\limits_{r=1}^n a_{rs}$ is the average of column $s$, and where $\overline{a_{\cdot\cdot}}=\frac{1}{n^2}\sum\limits_{r,s=1}^n a_{rs}$ is the average entry in the matrix.
 Since $\bD$ is a symmetric matrix, it follows that $\bA$ and $\bB$ are each symmetric, and therefore $\bB$ has $n$ real eigenvalues.

Assume for convenience that there are at least $m$ positive eigenvalues for matrix $\bB$, where $m\le n$.
By the spectral theorem of symmetric matrices, let $\bB = \bGamma \bLambda \bGamma^\top$ with $\bGamma$ containing unit-length eigenvectors of $\bB$ as its columns, and with the diagonal matrix $\bLambda$ containing the eigenvalues of $\bB$ in decreasing order along its diagonal.
Let $\bLambda_m$ be the $m\times m$ diagonal matrix of the largest $m$ eigenvalues sorted in descending order, and let $\bGamma_m$ be the $n\times m$ matrix of the corresponding $m$ eigenvectors in $\bGamma$.
Then the coordinates of the MDS embedding into $\R^m$ are given by the $n\times m$ matrix $\bX=\bGamma_m\bLambda_m^{1/2}$.
More precisely, the MDS embedding consists of the $n$ points in $\R^m$ given by the $n$ rows of $\bX$.

The procedure for classical MDS can be summarized in the following steps.
\begin{enumerate}
\item Compute the matrix $\bA = (a_{ij})$, where $a_{ij} = -\frac{1}{2}d^2_{ij}$.
\item Apply double-centering to $\bA$:
Define $\bB = \bH\bA\bH$, where $\bH = \bI - n^{-1}\mathbf{1}\mathbf{1}^\top$.
\item Compute the eigendecomposition of $\bB = \bGamma \bLambda \bGamma^\top$.
\item  Let $\bLambda_m$ be the matrix of the largest $m$ eigenvalues sorted in descending order, and let $\bGamma_m$ be the matrix of the corresponding $m$ eigenvectors.
Then, the coordinate matrix of classical MDS is given by $\bX=\bGamma_m\bLambda_m^{1/2}$.
\end{enumerate}

The following is a fundamental criterion in determining whether the dissimilarity matrix $\bD$ is Euclidean or not.
\begin{theorem}~\cite[Theorem~14.2.1]{bibby1979multivariate}\label{MDS p.s.d}
Let $\bD$ be a dissimilarity matrix, and define $\bB$ by equation~\eqref{B definition}.
Then $\bD$ is Euclidean if and only if $\bB$ is a positive semi-definite matrix.
\end{theorem}

In particular, if $\bB$ is positive semi-definite of rank $m$, then a perfect realization of the dissimilarities can be found by a collection of points in $m$-dimensional Euclidean space. 

If we are given a Euclidean matrix $\bD$, then the classical solution to the MDS problem in $k$ dimensions has the following optimal property:
\begin{theorem}~\cite[Theorem~14.4.1]{bibby1979multivariate}
Let $\bD$ be a Euclidean distance matrix corresponding to a
configuration $\bX$ in
$\R^m$, and fix $k$ $(1 \leq k \leq m)$.
Then amongst all projections
$\mathbf{XL_1}$ of $\bX$ onto $k$-dimensional subspaces of $\R^m$, the quantity $\sum\limits_{r,s=1} ^n (d_{rs}^2 - \hat  d _{rs}^2)$ is
minimized when $\bX$ is projected onto its principal coordinates in $k$ dimensions.
\end{theorem}

This theorem states that in this setting, MDS minimizes the sum of squared errors in distances, over all possible projections.
In the following paragraph, we show that an analogous result is true for the sum of squared errors in inner-products, i.e., for the loss function $\strain$.

Let $\bD$ be a dissimilarity matrix, and let $\bB = \bH\bA\bH$.
A measure of the goodness of fit of MDS, even in the case when $\bD$ is not Euclidean, can be obtained as follows~\cite{bibby1979multivariate}.
If $\hat{\bX}$ is a fitted configuration in $\R^m$ with centered inner-product matrix $\hat{\bB}$, then a measure of the discrepancy between $\bB$ and $\hat{\bB}$ is the following $\strain$ function~\cite{mardia1978some},
\begin{equation}\label{eq:optimization}
\mathrm{tr}((\bB-\hat{\bB})^2)=\sum\limits_{i,j=1}^n(b_{i,j}-\hat{b}_{i,j})^2.
\end{equation}

\begin{theorem}~\cite[Theorem~14.4.2]{bibby1979multivariate}\label{thm: strain-minimization-cMDS}
Let $\bD$ be a dissimilarity matrix (not necessarily Euclidean).
Then for fixed $m$, the $\strain$ function in \eqref{eq:optimization} is minimized over all configurations $\hat {\bX}$ in $m$
dimensions when $\hat {\bX}$ is the classical solution to the MDS problem.
\end{theorem}

The reader is referred to~\cite[Section~2.4]{cox2000multidimensional} for a summary of a related optimization with a different normalization, due to Sammon~\cite{sammon1969nonlinear}.

\subsection{Distance Scaling} 
Some popular metric MDS methods that attempt to minimize distances (versus inner-products) include
\begin{itemize}
\item least squares scaling, which minimizes a variation of loss functions, and
\item metric Scaling by Majorising a Complicated Function (SMACOF), which minimizes a form of the $\stress$ function.
\end{itemize}
The reader is referred to~\cite{cox2000multidimensional} for a description of least squares scaling and to~\cite{borg2005modern, cox2000multidimensional} for the theory of SMACOF and The Majorisation Algorithm.

\section{Non-Metric Multidimensional Scaling}
Non-Metric Multidimensional Scaling assumes that only the ranks or orderings of the dissimilarites are known.
Hence, this method produces a map which tries to reproduce these ranks and not the observed or actual dissimilarities. 
Thus, only the ordering of the dissimilarities is relevant to the methods of approximations. 
In~\cite[Chapter 3]{cox2000multidimensional}, the authors present the underlying theory of non-metric multidimensional
scaling developed in the 1960s, including Kruskal's method.
Other methods that fall under Non-Metric Multidimensional Scaling include Non-Metric SMACOF and Sammon Mapping.

\section{Simple Examples: Visualization}

In this section, we consider three simple dissimilarity matrices (input of MDS) and their Euclidean embeddings in $\R^2$ or $\R^3$ (output of MDS). 
The first two are Euclidean distance matrices, whereas the third is non-Euclidean.

\begin{figure}[h] 
    \subfloat[MDS embedding of $D_1$ into $\R^2$.] {\label{fig1:sub:subfigure1}\includegraphics[width=0.55\textwidth]{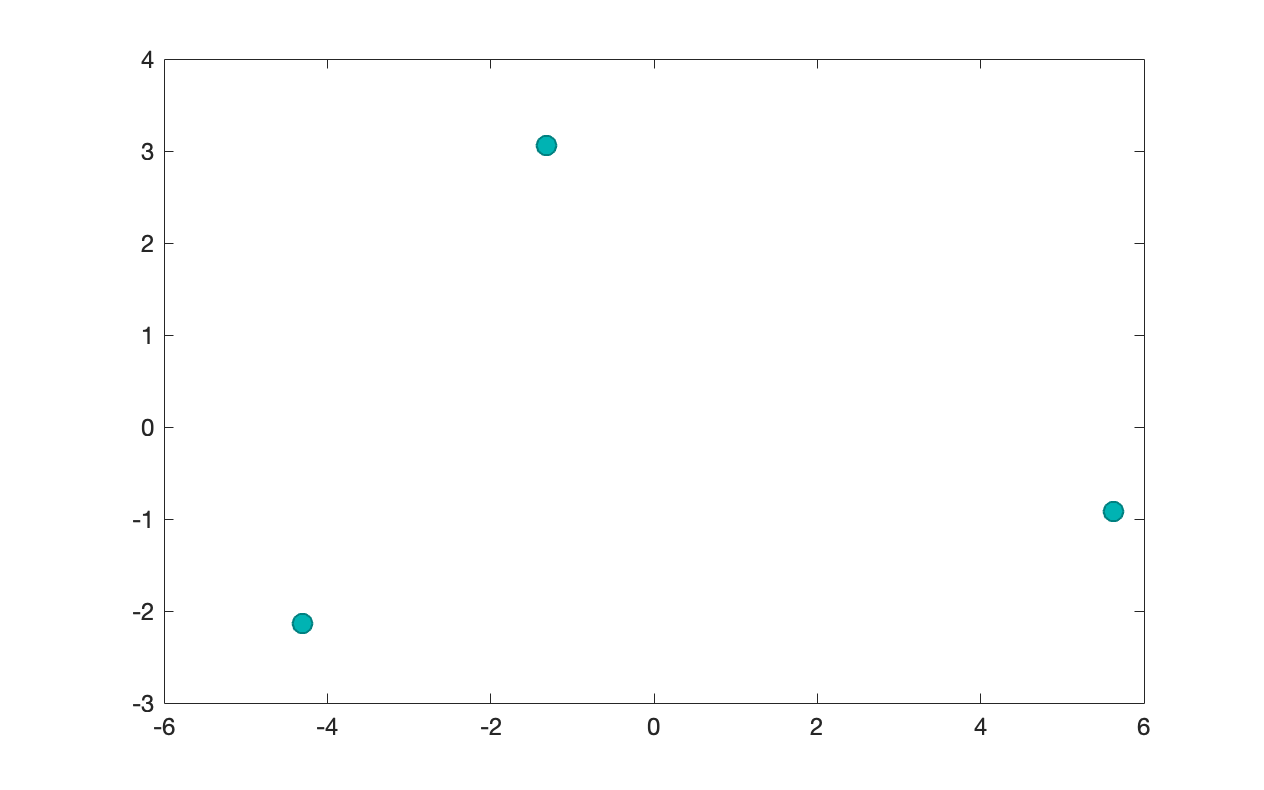}} \\
    
    \subfloat[MDS embedding of $D_2$ into $\R^3$.]{\label{fig1:sub:subfigure2}\includegraphics[width=0.55\textwidth]{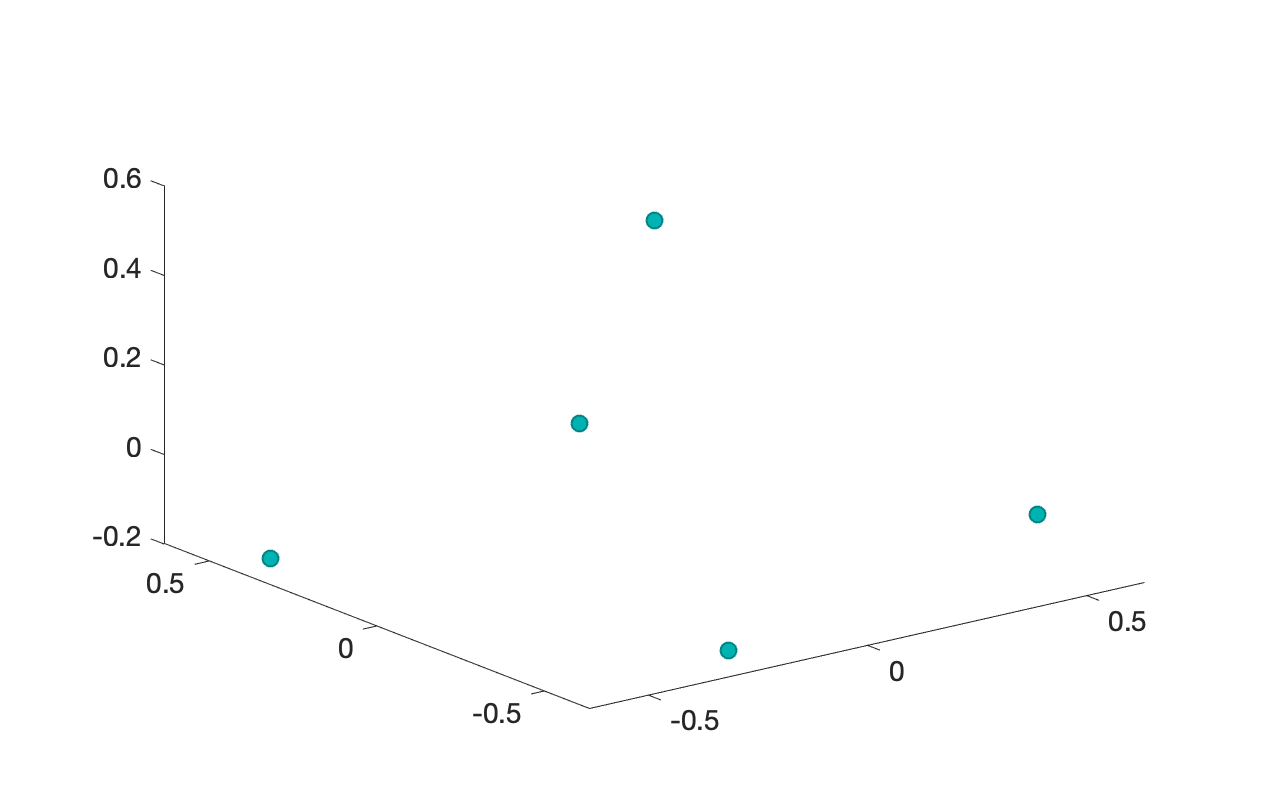}}
    
    \subfloat[MDS embedding of $D_3$ into $\R^2$.]{\label{fig1:sub:subfigure3}\includegraphics[width=0.55\textwidth]{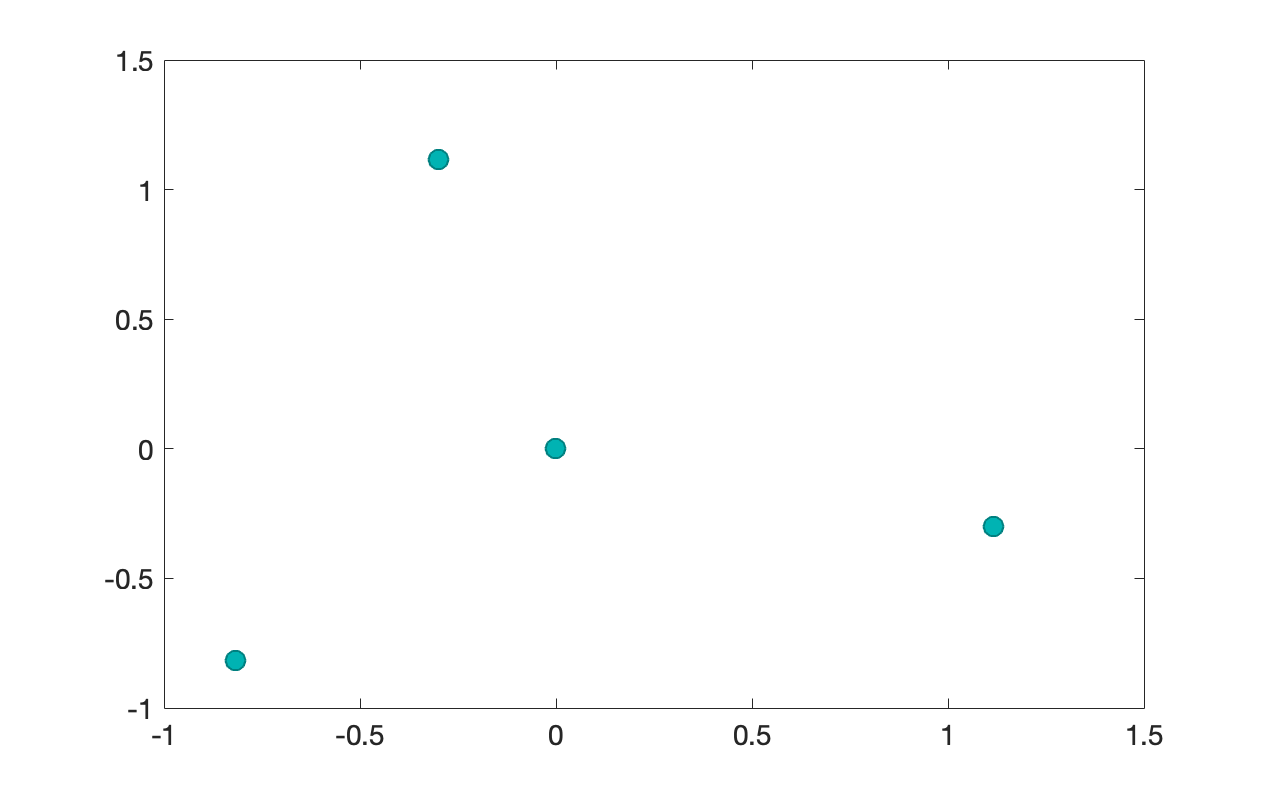}}

    \caption{MDS embeddings into $\R^2$ or $\R^3$ of the three dissimilarity matrices $D_1$, $D_2$, and $D_3$.}

    \label{fig: fig1} 

\end{figure}

\begin{enumerate}
\item Consider the following dissimilarity matrix, 
\[D_1 = \begin{pmatrix}
 0 & 6 & 8  \\
 6 & 0  & 10  \\
 8 & 10 & 0
\end{pmatrix},\]
which is an example of a dissimilarity matrix that can be isometrically embedded in $\R^2$ (Figure~\ref{fig1:sub:subfigure1}) but not in $\R^1$.

\item Consider the following dissimilarity matrix, \[ D_2 = \begin{pmatrix}
 0 & 1& 1 &\sqrt 2 & 1 \\
 1 & 0  & \sqrt 2 & 1 & 1  \\
 1 & \sqrt 2 & 0 & 1 & 1\\
 \sqrt 2 & 1 & 1  & 0 & 1\\
 1 & 1 & 1 & 1 & 0\\
\end{pmatrix},\] which is an example of a dissimilarity matrix that can be isometrically embedded in $\R^3$ (Figure~\ref{fig1:sub:subfigure2}) but not in $\R^2$.

\item Consider the following dissimilarity matrix, \[ D_3 = \begin{pmatrix}
 0 & 2& 2 & 1 \\
 2 & 0  &  2  & 1  \\
 2 &  2 & 0 & 1\\
1 & 1 & 1  & 0
\end{pmatrix},\]
which is an example of a dissimilarity matrix that cannot be isometrically embedded into any Euclidean space.
Indeed, label the points in the metric space $x_1, x_2, x_3, x_4$ in order of their row/column in $D_3$.
In any isometric embedding in $\R^n$, the points $x_1, x_2, x_3$ would need to get mapped to an equilateral triangle.
Then the point $x_4$ would need to get mapped to the midpoint of each edge in this triangle, which is impossible in Euclidean space.
Figure~\ref{fig1:sub:subfigure3} shows the embedding of this metric space in $\R^2$.
\end{enumerate}

\chapter{Operator Theory}
\label{chap: Operator Theory}

This section serves to inform the reader on some of the concepts in infinite-dimensional linear algebra and operator theory used throughout our work.

\section{Kernels and Operators}
We denote by $L^2(X, \mu)$  the set of square integrable $L^2$-functions with respect to the measure $\mu$. 
We note that $L^2(X, \mu)$ is furthermore a Hilbert space, after equipping it with the inner product given by
\[\langle f,g \rangle = \int \limits_X fg\ d\mu.\]

\begin{definition} A measurable function $f$ on $X \times X$ is said to be \emph{square-integrable} if satisfies the following three conditions~\cite{smithies1970, kanwal2013linear}: 
\begin{itemize}
\item $f(x,s)$ is a measurable function of $(x,s)\in X \times X$, with
\[ \int\limits_X \int\limits_X |f(x,s)|^2 \mu (\diff x) \mu (\diff s) < \infty; \]
\item for each $s\in X$, the function $f(x,s)$ is a measurable function in $x$, with
\[ \int\limits_X |f(x,s)|^2 \mu (\diff x) < \infty; \]
\item for each $x\in X$, the function $f(x,s)$ is a measurable function in $s$, with
\[ \int\limits_X |f(x,s)|^2 \mu (\diff s)< \infty .\]
\end{itemize}
\end{definition}

The $L^2$-norm of a square-integrable function is given by
\[\int \limits_X \int \limits_X |f(x,s)| ^2 \mu(\diff x) \mu(\diff s) < \infty.\] 
We denote by $L^2_{\mu \otimes \mu }(X\times X)$ the set of square integrable functions with respect to the measure $\mu \otimes \mu$. 

\begin{definition}
A set $\{f_i\}_{i \in \N}$ of real-valued functions $f_i \in L^2(X, \mu)$ is said to be \emph{orthonormal} if \[ \langle f_i , f_j \rangle = \delta_{ij}. \]
\end{definition}

When it is clear from the context, we will simply write $L^2(X)$ and $L^2(X\times X)$ instead of $L^2(X, \mu)$ and $L^2_{\mu \otimes \mu }(X\times X)$, respectively.

In this context, a \emph{real-valued $L^2$-kernel} $K \colon X \times X \to \R$ is a continuous measurable square-integrable function i.e.\ $K \in L^2_{\mu \otimes \mu }(X\times X)$.
Most of the kernels that we define in our work are symmetric.

\begin{definition}
A kernel $K$ is \emph{symmetric} (or \emph{complex symmetric} or \emph{Hermitian}) if 
\[ K(x,s) = \overline{K(s,x)} \quad\mbox{for all }x,s\in X, \]
where the overline denotes the complex conjuguate.
In the case of a real kernel, the symmetry reduces to the equality
\[ K(x,s) = K(s,x).\]
\end{definition} 

\begin{definition}\label{def: psdk}
A symmetric function ${\displaystyle K:{X}\times {X}\to \mathbb {R} }$ is called a \emph{positive semi-definite kernel} on $X$ if
\[{\sum\limits _{i=1}^{n}\sum\limits _{j=1}^{n}c_{i}c_{j}K(x_{i},x_{j})\geq 0}\]
holds for any $m\in \N$, any $x_1,\dots ,x_m\in X$, and any $c_1,\dots,c_m\in \R$.
\end{definition}

We remark that the above definition is equivalent to saying that for all vectors $c \in \R^m$, we have
\[c^T K c \geq 0, \quad \mbox{where } K_{ij} = K(x_i, x_j).\]
At least in the case when $X$ is a compact subspace of $\R^m$ (and probably more generally), we have the following equivalent definition.

\begin{definition}
Let $X$ be a compact subspace of $\R^m$, and let $K\in L^2_{\mu \otimes \mu }(X\times X)$ be a real-valued symmetric kernel.
Then $K(x,s)$ is a \emph{positive semi-definite kernel} on $X$ if
\[ \int _{X \times X} K(x, s)f(x)f(s)d(\mu \otimes \mu)(x,s) \geq 0 \]
for any $f \in L^2(X, \mu)$.
\end{definition}

\begin{definition}
Let $\C^{m \times n}$ denote the space of all $m\times n$ matrices with complex entries. We define an inner product on $\C^{m \times n}$ by 
\[(A,B) = \tr(A^*B), \]
where $\tr$ denotes the trace and $^*$ denotes the Hermitian conjugate of a matrix. If $A = (a_{ij})$ and $B = (b_{ij})$, then 
\[ (A,B) = \sum\limits_{i=1}^m \sum\limits_{j=1}^n \overline{{a}_{ij}}b_{ij}.\]
The corresponding norm,
\[\| A \| = \left( \sum\limits_{i=1} ^{m} \sum\limits_{j=1} ^{n} {|a_{ij}|}^2 \right)^{1/2},\]
is called the \emph{Hilbert--Schmidt norm}.
\end{definition}

For $\cH_1$ and $\cH_2$ Hilbert spaces, we let $\cB(\cH_1, \cH_2)$ denote the set of bounded linear operators from $\cH_1$ to $\cH_2$.
Similarly, for $\cH$ a Hilbert space, we let $\cB(\cH)$ denote the set of bounded linear operators from $\cH$ to itself.

\begin{definition}
Let $T \in \cB(\cH_1, \cH_2)$. Then there exists a unique operator from $\cH_2$ into $\cH_1$, denoted by $T^*$ (or $T'$), such that
\[\langle Tx,y\rangle = \langle x,T^{*}y \rangle \quad \forall x \in \cH_1,y \in \cH_2.\]
This operator $T^*$ is called the \emph{adjoint} of $T$.
\end{definition}

\begin{definition}
Suppose $T \in \cB(\cH)$. If $T = T^*$, then $T$ is called \emph{self-adjoint}.
\end{definition}

\begin{definition}\label{def: trace-class}
A bounded linear operator $A\in\cB(\cH)$ over a separable Hilbert space $\cH$ is said to be in the \emph{trace class} if for some (and hence all) orthonormal bases $\{e_k\}_{k \in \N}$ of $\cH$, the sum of positive terms
\[\|A\|_{1}= \tr |A| =\sum _{k}\langle (A^{*}A)^{1/2}\,e_{k},e_{k}\rangle \]
is finite. In this case, the trace of A, which is given by the sum
\[ \tr (A)=\sum _{k}\langle Ae_{k},e_{k}\rangle, \]
is absolutely convergent and is independent of the choice of the orthonormal basis.
\end{definition}

\begin{definition} \label{def: HS-norm}
A \emph{Hilbert--Schmidt operator} is a bounded linear operator $A$ on a Hilbert space $\cH$ with finite Hilbert--Schmidt norm
\[{\|A\|_{HS}^{2}={\tr}(A^{*}A) =\sum\limits _{i\in I}\|Ae_{i}\|^{2}}, \]
where $\| \cdot \|$ is the norm of $\cH$, $ \{e_{i}:i\in I\}$ is an orthonormal basis of $\cH$, and $\tr$ is the trace of a nonnegative self-adjoint operator. 
\end{definition}

Note that the index set need not be countable.
This definition is independent of the choice of the basis, and therefore
\[ \|A\|_{HS}^{2}=\sum\limits _{i,j}|A_{i,j}|^{2}, \]
where  $A_{i,j}=\langle e_{i},Ae_{j}\rangle$. In Euclidean space, the Hilbert--Schmidt norm is also called the Frobenius norm.

When $\cH$ is finite-dimensional, every operator is trace class, and this definition of the trace of $A$ coincides with the definition of the trace of a matrix.

\begin{proposition}\label{prop:trace-conj}
If $A$ is bounded and $B$ is trace class (or if  $A$ and $B$ are Hilbert--Schmidt), then $AB$ and $BA$ are also trace class, and $\tr(AB) = \tr(BA)$.
\end{proposition}

\begin{definition}[Hilbert--Schmidt Integral Operator]\label{def: Kernel Operator}
Let $(X, \Omega, \mu)$ be a $\sigma$-finite measure
space, and let $K \in L^2_{\mu \otimes \mu}(X \times X)$.
Then the integral operator
\[ [T_K\phi](x) = \int \limits_X K(x,s) \phi(s) \mu(\diff s)\]
defines a linear mapping acting from the space $L^2 (X, \mu)$ into itself.
\end{definition}

Hilbert--Schmidt integral operators are both continuous (and hence bounded) and compact operators.
 
Suppose $ \{e_{i}\}_{i\in \N}$ is an orthonormal basis of $L^2(X, \mu)$.  Define $e_{nm}(x,s) = e_n(s) e_m(x)$. Then, $({e_{mn}})_{m,n \in \N}$ forms an orthonormal basis of $L^2_{\mu \otimes \mu}(X \times X)$.
 
\begin{proposition} \label{prop: HS-kernel norm}
$ \| T_K \|^2 _{HS} = \| K  \|^2 _{L^2(X \times X)}$.
\end{proposition}

\begin{proof} By Definition~\ref{def: HS-norm}, we have
\begin{align*}
\| T_K \|^2 _{HS} & = \sum \limits _n \| T_K e_n \|_{L^2(X)} \\
& = \sum\limits_{n,m} | {\langle e_m, T_K e_n \rangle }_{L^2(X)} | ^2\\
& = \sum\limits_{n,m} \left| \left\langle e_m, \int K(x,s)e_n(x) \mu (\diff x) \right\rangle _ {L^2(X)}  \right|^2 \\
& = \sum \limits_{n,m} \left| \int\limits_X \int \limits_X K(x,s) e_n(x) e_m(s) \mu (\diff x) \mu (\diff s)  \right|^2 \\
& = \sum \limits _{n,m} |\langle K, e_{nm} \rangle  _{L^2 (X \times X)} |^2 \\
& =  \| K  \|^2 _{L^2(X \times X)}.
\end{align*}
\end{proof}

\begin{definition}
A Hilbert--Schmidt integral operator is a \emph{self-adjoint operator} if and only if $K(x,y) = \overline{K(y,x)}$ (i.e\ $K$ is a symmetric kernel) for almost all $(x,y) \in X \times X$ (with respect to $\mu \times \mu$).
\end{definition}

\begin{remark}\label{rem: relation to Strain}
Suppose $T_K$ a self-adjoint Hilbert-Schmidt integral operator. By Definition~\ref{def: HS-norm}, we have  \[ \| T_K \|^2 _{HS} = {\tr}(T_K^{*}T_K) = {\tr}((T_K)^2),\] and by Proposition~\ref{prop: HS-kernel norm} \[\| T_K \|^2_{HS} = \| K  \|^2 _{L^2(X \times X)} = \int \limits _X \int \limits _X  |K(x,s)| ^2  \mu (\diff x) \mu (\diff s).\]
These properties will be useful in defining a $\strain$ function for MDS of infinite metric measure spaces in Section~\ref{sec: strain minimization}.
\end{remark}

\begin{definition}\label{p.s.d.o}
A bounded self-adjoint operator $A$ on a Hilbert space $\cH$ is called a \emph{positive semi-definite operator} if $\langle Ax, x \rangle \geq 0$
for any $x \in \cH$.
\end{definition}

It follows that for every positive semi-definite operator $A$, the inner product $\langle Ax, x \rangle$ is real for every $x \in \cH$.
Thus, the eigenvalues of $A$, when they exist, are real.

\begin{definition}
A linear operator $T \colon L^2(X, \mu) \to L^2(X, \mu)$ is said to be \emph{orthogonal} (in \R) or \emph{unitary} (in \C) if $\langle f, g \rangle = \langle T(f) , T(g) \rangle$ for all $f$ and $g$ in $L^2(X, \mu)$.
\end{definition}

\begin{definition}\label{def:otimes}
Suppose $h, g \in L^2(X, \mu)$. Define the linear operator $h \otimes g \colon L^2(X, \mu) \to L^2(X, \mu)$ as follows
\[(h \otimes g)(f) = \langle g,f \rangle h.\]
Then $h \otimes g$ is a Hilbert--Schmidt integral operator associated with the $L^2$-kernel $K(x,s)=h(x)g(s)$.
\end{definition}

Suppose $\{e_n\}_{n\in \N}$ and $\{f_n\}_{n\in \N}$ are two distinct orthonormal bases for $L^2(X, \mu)$.
Define the linear operator $G \colon L^2(X, \mu) \to L^2(X, \mu)$ by 
\[G(\cdot) =\sum\limits_n f_n\otimes e_n =\sum\limits_n\langle e_n,\cdot \rangle f_n .\]

\begin{proposition}
$G$ is an orthogonal operator.
\end{proposition}

\begin{proof} For any $\phi, \psi \in L^2(x,\mu)$, we have
\begin{align*}
\langle G (\phi) , G (\psi) \rangle 
&= \left\langle \sum_i \langle e_i, \phi \rangle f_i , \sum_j \langle e_j, \psi \rangle f_j \right\rangle \\
& =\sum _{i,j} \left\langle \langle e_i, \phi \rangle f_i, \langle e_j, \psi\rangle f_j \right\rangle \\
& = \sum _{i,j} \langle e_i, \phi \rangle \langle e_j, \psi \rangle \langle f_i , f_j \rangle \\
& = \sum _{i} \langle e_i, \phi \rangle \langle e_i ,\psi \rangle\\
& = \langle \phi , \psi \rangle.
\end{align*}
Therefore, $G$ is an orthogonal operator.
\end{proof}

More generally, if $\{e_n\}_{n \in \N}$ and $\{f_n\}_{n \in \N}$ are any bases for $L^2(X, \mu)$, then $G=\sum\limits_n f_n\otimes e_n$ is a change-of-basis operator (from the $\{e_n\}_{n \in \N}$ basis to the $\{f_n\}_{n \in \N}$ basis), satisfying $G(e_i)=f_i$ for all $i$.

\section{The Spectral Theorem}
\begin{definition}
A complex number $\lambda \in \C$ is an \emph{eigenvalue} of $T \in \cB(\cH)$
if there exists a non-zero vector $x \in \cH$ such that $T x = \lambda x$. The vector $x$ is called an
\emph{eigenvector} for $T$ corresponding to the eigenvalue $\lambda$. Equivalently, $\lambda$ is an eigenvalue
of $T$ if and only if $T - \lambda I$ is not one-to-one.
\end{definition}

\begin{theorem}[Spectral theorem on compact self-adjoint operators]
\label{thm:spectral}
Let $\cH$ be a not necessarily separable Hilbert space, and suppose $T \in \cB (\cH)$ is compact self-adjoint operator.
Then $T$ has at most a countable number of nonzero eigenvalues $\lambda_n  \in \R$, with a corresponding orthonormal set $\{ e_n \}$ of eigenvectors such that
\[ T(\cdot) = \sum\limits_n \lambda_n \langle e_n, \cdot \rangle\ e_n.\]
Furthermore, the multiplicity of each nonzero eigenvalue is finite, zero is the only possible accumulation point of $\{ \lambda_n \}$, and if the set of non-zero eigenvalues is infinite then zero is necessarily an accumulation point.
\end{theorem}

A fundamental theorem that characterizes positive semi-definite kernels is the Generalized Mercer's Theorem, which states the following:

\begin{theorem}~\cite[Lemma 1]{kuhn1987eigenvalues}\label{Thm: Mercer's Theorem}
Let $X$ be a compact topological Hausdorff space equipped with a finite Borel measure $\mu$, and let $K \colon X \times  X \to \C$ be a continuous positive semi-definite kernel.
Then there exists a scalar sequence $\{ \lambda_n \} \in \ell_1$ with $\lambda_1 \geq \lambda_2 \geq \cdots \geq 0$, and an orthonormal system $\{ \phi_n\}$ in $L^2 (X, \mu)$ consisting of continuous functions only,
such that the expansion 
\begin{equation}\label{eq: Mercer}
K(x, s) = \sum\limits _{n=1} ^\infty \lambda_n\phi_n(x)\phi_n(s), \quad x, s \in \supp(\mu)
\end{equation} converges uniformly.
\end{theorem}

Therefore, we have the following.
A Hilbert--Schmidt integral operator
\begin{equation}\label{eq: Kernel Operator}
[T_K\phi](x) = \int \limits_X K(x,s) \phi(s) \mu(\diff s)\end{equation}
associated to a symmetric positive semi-definite $L^2$-kernel $K$, is a positive-semi definite operator.
Moreover, the eigenvalues of $T_K$ can be arranged in non-increasing order, counting them according to their algebraic multiplicities:
$\lambda_1 \geq \lambda_2 \geq \ldots \geq 0$.

\chapter{MDS of Infinite Metric Measure Spaces}
\label{chap: iMDS}

Classical multidimensional scaling can be described either as a $\strain$-minimization problem, or as a linear algebra algorithm involving eigenvalues and eigenvectors.
Indeed, one of the main theoretical results for cMDS is that the linear algebra algorithm solves the corresponding $\strain$-minimization problem (see Theorem~\ref{thm: strain-minimization-cMDS}).
In this section, we describe how to generalize both of these formulations to infinite metric measure spaces.
This will allow us to discuss cMDS of the entire circle for example, without needing to restrict attention to finite subsets thereof.

\section{Proposed Approach}
Suppose we are given a bounded (possibly infinite) metric measure space $(X, d_X, \mu)$, where $d_X$ is a real-valued $L^2$-function on $X \times X$ with respect to the measure $\mu \otimes \mu$.
By \emph{bounded} we mean that $d_X$ is a bounded $L^2$-kernel, i.e.\ there exists a constant $C\in \R$ such that $d(x,s)\le C$ for all $x,s\in X$.
When it is clear from the context, the triple $(X, d_X, \mu)$ will be denoted by only
$X$.
Even if $X$ is not Euclidean, we hope that the metric $d_X$ can be approximately represented by a Euclidean metric $d_{\hat X}: \hat X \times \hat X \to \R$ on a space $\hat X\subseteq\ell^2$ or $\hat X\subseteq\R^m$, where perhaps the Euclidean space is of low dimension (often $m=2$ or $3$).

A metric space $(X, d_X)$ is said to be \emph{Euclidean} if $(X, d_X)$ can be isometrically embedded into $(\ell^2,\| \cdot \|_2)$.
That is, $(X, d_X)$ is Euclidean if there exists an isometric embedding $f\colon X\to \ell^2$, meaning $\forall x,s \in X$, we have that $d_X(x,s) = d_{\ell^2}(f(x),f(s))$.
Furthermore, we call a metric measure space $(X, d_X, \mu_X)$ \emph{Euclidean} if its underlying metric space $(X, d_X)$ is.
As will be discussed below, $X$ could also be Euclidean in the finite-dimensional sense, meaning that there is an isometric embedding $f\colon X\to \R^m$.

\section{Relation Between Distances and Inner Products} \label{ss: theoretical results} 
The following discussion is carried analogously to the arguments presented in~\cite{pekalska2001generalized} (for spaces of finitely many points). 
Consider the bounded metric measure space $(X, d, \mu)$ where $X \subseteq \ell ^2$  (or $\R^m$) and $d \colon X \times X \to \R$ the Euclidean distance.
We have the following relation between distances and inner-products in $\ell^2$ (or $\R^m$),
\[d^2(x,s) = \langle x-s, x-s \rangle = \langle x, x \rangle + \langle s, s \rangle -2 \langle x, s \rangle = \|x\|^2 + \|s\|^2 - 2  \langle x, s \rangle.\]

This implies
 \begin{equation} \label{eq: inner prod}
 \langle x, s \rangle = -\frac{1}{2} \left( d^2(x,s) - \|x\|^2 - \|s\|^2 \right).
\end{equation}

Furthermore, if we let let $\bar x = \int \limits_X w \ \mu (\diff w)$ denote the center of the space $X$, then
\begin{equation*}
\begin{split}
d^2(x,\bar x)
= & \langle x, x \rangle + \langle \bar x, \bar x \rangle -2 \langle x, \bar x \rangle  \\
= & \|x\|^2 + \int \limits_X \int \limits_X \langle w, z \rangle\ \mu (\diff z) \mu (\diff w) - 2 \int \limits_X \langle x, z \rangle\ \mu (\diff z) \\
= & \|x\|^2 + \frac{1}{2} \int \limits_X \int \limits_X  \left( \|w\|^2 + \|z\|^2 - d^2(w,z) \right) \mu (\diff z) \mu (\diff w)  \\ & - \int \limits_X \left( \|x\|^2 + \|z\|^2 - d^2(x,z) \right) \mu (\diff z)\\ 
= & \int \limits_X  d^2(x,z)\ \mu (\diff z) - \frac{1}{2} \int \limits_X \int \limits_X d^2(w,z)\ \mu (\diff z) \mu (\diff z).
\end{split}
\end{equation*}

Without loss of generality, we can assume our space $X$ is centered at the origin (i.e $\bar x = \mathbf 0$).
This implies that $\|x\|^2 = d^2(x, \mathbf 0) = d^2(x, \bar x)$.
Therefore, equation~\eqref{eq: inner prod} can be written as
\begin{equation}
\begin{split}
\langle x, s \rangle 
& = -\frac{1}{2} \left( d^2(x,s) - \int \limits_X d^2(x,z)\ \mu (\diff z )- \int \limits_X d^2(w,s)\ \mu (\diff w) + \int \limits_X \int \limits_X d^2(w,z)\ \mu (\diff w) \mu (\diff z) \right).
\end{split}
\end{equation}

Now, let $(X, d, \mu)$ be any metric measure space where $d$ is an $L^2$-function on $X \times X$ with respect to the measure $\mu \otimes \mu$.
Define $K_A(x,s) = -\frac{1}{2}d^2(x,s) $ and define $K_B$ as
\[K_B(x,s)= K_A(x,s)- \int \limits_X  K_A(x,z)\ \mu(\diff z) - \int \limits_X  K_A(w,s)\ \mu(\diff w) + \int_{X \times X} K_A(w,z)\ \mu(\diff w \times \diff z).\]

\begin{proposition}\label{prop:dist-inner-identity}
Given a metric measure space $(X,d,\mu)$, where $d \in L^2_{\mu \otimes \mu}(X \times X)$, construct $K_A$ and $K_B$ as defined above.
Then we have the relation
\[d^2(x,s) = K_B(x,x) + K_B(s,s) -2K_B(x,s).\]
\end{proposition}

\begin{proof} We have
\[
\begin{split}
&  K_B(x,x) + K_B(s,s) -2K_B(x,s) \\
& =  K_A(x,x) - \int \limits_X K_A(x,z) \mu( \diff z) - \int \limits_X K_A(w,x) \mu (\diff w)  + \int_{X \times X} K_A(w,z) \mu (\diff w) \mu (\diff z) \\ &  \ + K_A(s,s) - \int \limits_X K_A(s,z)\mu (\diff z) -\int \limits_X K_A(w,s) \mu (\diff w)  + \int_{X \times X} K_A(w,z)\mu (\diff w) \mu (\diff z) \\ & \ -2K_A(x,s) +2\int \limits_X K_A(x,z) \mu (\diff z) +2\int \limits_X K_A(w,s) \mu (\diff w) -2\int_{X \times X} K_A(w,z) \mu (\diff w) \mu (\diff z)
\\
& =  -2K_A(x,s) = d^2(x,s).
\end{split}
\]
Indeed, the intermediate steps follow since $K_A(x,x)=K_A(s,s)=0$, and since by the symmetry of $K_A$ we have $\int \limits_X K_A(x,z) \mu (\diff z ) = \int \limits_X K_A(w,x) \mu (\diff w)$ and $\int \limits_X K_A(s,z) \mu (\diff z) = \int \limits_X K_A(w,s) \mu (\diff w)$.
\end{proof}

\section{MDS on Infinite Metric Measure Spaces}\label{ss:infinite-mds}
In this section, we explain how multidimensional scaling generalizes to possibly infinite metric measure spaces that are bounded.
We remind the reader that by definition, all of the metric measure spaces we consider are equipped with probability measures.

Let $(X, d, \mu)$ be a bounded metric measure space, where $d$ is a real-valued $L^2$-function on $X \times X$ with respect to the measure $\mu \otimes \mu$.
We propose the following MDS method on infinite metric measure spaces:

\begin{enumerate}[(i)]
\item From the metric $d$, construct the kernel $K_A \colon X \times X \to \R$ defined as $K_A(x,s)= -\frac{1}{2}d^2(x,s)$.
\item Obtain the kernel $K_B \colon X \times X \to \R$ defined as 
\begin{equation}\label{eq : kernel B}
K_B(x,s)= K_A(x,s) - \int \limits_X  K_A(w,s) \mu(\diff w) - \int \limits_X  K_A(x,z) \mu(\diff z) + \int_{X \times X} K_A(w,z) \mu(\diff w \times \diff z).
\end{equation}
Assume $K_B \in L^2(X \times X)$.
Define $T_{K_B}\colon L^2(X) \to L^2(X)$ as
\[[T_{K_B}\phi](x) = \int \limits_X K_B(x,s) \phi(s) \mu(\diff s).\]

Note that, kernels $K_A$ and $K_B$ are symmetric, since $d$ is a metric.

\item Let $\lambda_1\geq \lambda_2 \geq \dots$ denote the eigenvalues of $T_{K_B}$ with corresponding eigenfunctions  $\phi_1, \phi_2, \ldots$, where the $\phi_i \in L^2(X)$ are real-valued functions. Indeed, $\{\phi_i\}_{i \in \N}$ forms an orthonormal system of $L^2(X)$.

\item Define $K_{\hat B}(x,s) = \sum\limits_{i=1} ^ \infty \hat \lambda_i \phi _i (x) \phi_i (s)$, where
\[ \hat \lambda_i =  \begin{cases} 
\lambda_i &\text{if } \lambda_i \geq 0, \\
0 &\text{if } \lambda_i < 0.\\
\end{cases}
\]
Define $T_{K_{\hat B}}\colon L^2(X) \to L^2(X)$ to be the Hilbert--Schmidt integral operator associated to the kernel $K_{\hat B}$.
Note that the eigenfunctions $\phi _i$ for $T_{K_B}$ (with eigenvalues $\lambda_i$) are also the eigenfunctions for $T_{K_{\hat B}}$ (with eigenvalues $\hat \lambda_i$).
By Mercer's Theorem (Theorem~\ref{Thm: Mercer's Theorem}), $K_{\hat B}$ converges uniformly. 

\item Define the MDS embedding of $X$ into $\ell^2$ via the map $f\colon X\to \ell^2$ given by 
\[f(x)=\left(\sqrt{\hat \lambda_1} \phi_1(x),\sqrt{\hat \lambda_2} \phi_2(x),\sqrt{\hat \lambda_3} \phi_3(x),\ldots\right)\]
for all $x\in X$.
We denote the MDS embedding $f(X)$ by $\hat X$.

Similarly, define the MDS embedding of $X$ into $\R^m$ via the map $f_m\colon X\to \R^m$ given by
\[f_m(x)=\left(\sqrt{\hat \lambda_1} \phi_1(x), \sqrt{\hat \lambda_2} \phi_2(x), \ldots, \sqrt{\hat \lambda_m} \phi_m(x)\right)\]
for all $x\in X$.
We denote the image of the MDS embedding by $\hat X_m=f_m(X)$.

\item Define the measure $\hat \mu $ on $(\hat X, \hat d)$ to be the push-forward measure of $\mu$ with respect to map $f$, where $\cB(\hat X)$ is the Borel $\sigma$-algebra of $\hat X$ with respect to topology induced by $\hat d$, the Euclidean metric in $\R ^m$ or the $\ell^2$ norm.
Indeed, the function $f$ is measurable by Corollary~\ref{measurable} below.
\end{enumerate}

\begin{proposition}
The MDS embedding map $f\colon X\to \ell^2$ defined by
\[f(x)=\left(\sqrt{\hat \lambda_1} \phi_1(x),\sqrt{\hat \lambda_2} \phi_2(x),\sqrt{\hat \lambda_3} \phi_3(x),\ldots\right)\]
is a continuous map.
\end{proposition}

\begin{proof}
Define the sequence of embeddings $g_m \colon X \to \ell^2$, for $m \in \N$, by 
\[ g_m(x)=\left(\sqrt{\hat \lambda_1}\phi_1(x),\ldots, \sqrt{\hat \lambda_m}\phi_m(x), 0, 0,\ldots\right).\]
By Mercer's Theorem (Theorem~\ref{Thm: Mercer's Theorem}), we have that the eigenfunctions $\phi_i \colon X \to \R$ are continuous for all $i \in \N$.
It follows that $g_m$ is a continuous map for any $m< \infty$.

Consider the sequence of partial sums 
\[K_m(x,x) = \sum\limits _{i=1}^m \hat \lambda_i \phi_i ^2(x).\]
By Mercer's Theorem, $K_m(x,x)$ converges uniformly to $K(x,x)  = \sum\limits _{i=1}^ \infty \hat \lambda_i \phi_i ^2(x)$, as  $m \to \infty$.
Therefore, for any $\epsilon > 0$, there exists some $N(\epsilon) \in \N$ such that for all $m \ge N ( \epsilon)$,
\[ \| g_m(x) - f(x)\|_{2}^2 = \sum\limits _ {i = m+1} ^ \infty \hat \lambda_i \phi_i ^2 (x) = |K_m(x,x) - K(x,x) | < \epsilon \mbox{ for all } x \in X.\]
Therefore, $g_m$ converges uniformly to $f$ as $m \to \infty$.
Since the uniform limit of continuous functions is continuous, it follows that $f\colon X\to \ell^2$ is a continuous map.
\end{proof}

With respect to Borel measures every continuous map is measurable, and therefore we immediately obtain the following corollary.

\begin{corollary}\label{measurable}
The MDS embedding map $f\colon X\to \ell^2$ is measurable. 
\end{corollary}

\begin{theorem}\label{thm:inf-mds-euc}
A metric measure space $(X, d_X, \mu)$ is Euclidean if and only if $T_{K_B}$ is a positive semi-definite operator on $L^2(X, \mu).$
\end{theorem}

\begin{proof} Suppose $(X, d_X, \mu)$ is a Euclidean metric measure space.
For any discretization $X_n$ of $X$ and for any $x_i , x_j \in X_n$, the matrix $\bB _{ij} = K_B(x_i, x_j)$ is a positive semi-definite matrix by Theorem~\ref{MDS p.s.d}.
Furthermore, by Mercer's Theorem (Theorem~\ref{Thm: Mercer's Theorem}) we have that $T_{K_B}$ is a positive semi-definite operator on $L^2(X, \mu)$.

Now, suppose that $T_{K_B}$ is a positive semi-definite operator on $L^2(X, \mu)$.
By Mercer's Theorem (Theorem~\ref{Thm: Mercer's Theorem}), $K_B$ is a positive semi-definite kernel, and furthermore $K_B (x, s) = \sum\limits _{i=1} ^\infty \lambda_i\phi_i(x)\phi_i(s)$ converges uniformly in $L^2_{\mu \otimes \mu }(X \times X )$.
Thus, by Proposition~\ref{prop:dist-inner-identity} we have
\[d_X^2(x,s) = K_B(x,x) + K_B(s,s) -2K_B(x,s) = \sum\limits _{i=1} ^\infty \lambda_i(\phi_i(x) - \phi_i(s))^2= \| f(x) - f(s) \|_2^2.\]
Therefore, $(X, d_X, \mu)$ is a Euclidean measure metric space.
\end{proof}

\section{Strain Minimization}~\label{sec: strain minimization} 
We begin this section by defining a loss function, which in particular is a $\strain$ function. We then show that the MDS method for metric measure spaces described in Section~\ref{ss:infinite-mds} minimizes this loss function.

Define the \emph{\strain} function of $f$ as follows,
\[\strain(f) =\|T_{K_B}-T_{K_{\hat{B}}}\|_{HS}^{2}=\tr ((T_{K_B}-T_{K_{\hat{B}}})^2) = \int \int \left( K_B(x,t) - K_{\hat B}(x,t) \right)^2 \mu(dt) \mu(\diff x),\]
which is well-defined (see Remark~\ref{rem: relation to Strain}).

In order to show that MDS for metric measure spaces minimizes the \emph{\strain}, we will need the following two lemmas.

\begin{lemma}\label{lem: closest-nonneg-l2}
Let $\lambda\in\ell^2$, let $S=\{\hat{\lambda}\in \ell^2~|~\hat{\lambda}_i\ge0\text{ for all }i\}$, and define $\bar{\lambda}\in S$ by
\[ \bar \lambda_i =
\begin{cases} 
\lambda_i & \text{if }\lambda_i \geq 0 \\
0 & \text{if }\lambda_i < 0.
\end{cases}\]
Then $\| \lambda - \bar \lambda\|_2 \leq \| \lambda - \hat \lambda\|_2$ for all $\hat{\lambda}\in S$.
\end{lemma}

\begin{lemma}\label{lem: closest-nonneg-Rm}
Let $\lambda\in\ell^2$ have sorted entries $\lambda_1 \geq \lambda_2 \geq \ldots$, let $m\ge 0$, let \[S_m=\{\hat{\lambda}\in \ell^2~|~\hat{\lambda}_i\ge0\text{ for all }i\text{ with at most }m\text{ entries positive}\},\] and define $\bar{\lambda}\in S_m$ by
\[ \bar \lambda_i =
\begin{cases} 
\lambda_i & \text{if }\lambda_i \geq 0\text{ and }i\le m \\
0 & \text{if }\lambda_i < 0\text{ or }i>m.
\end{cases}\]
Then $\| \lambda - \bar \lambda\|_2 \leq \| \lambda - \hat \lambda\|_2$ for all $\hat{\lambda}\in S_m$.
\end{lemma}

The proofs of these lemmas are straightforward and hence omitted.

The following theorem generalizes~\cite[Theorem~14.4.2]{bibby1979multivariate}, or equivalently~\cite[Theorem~2]{trosset1997computing}, to the infinite case.
Our proof is organized analogously to the argument in~\cite[Theorem~2]{trosset1997computing}.

\begin{theorem}\label{Thm: infinite-mds-optimization}
Let $(X, d, \mu)$ be a bounded (and possibly non-Euclidean) metric measure  space.
Then $\strain(f)$ is minimized over all maps $f\colon X\to \ell^2$ or $f\colon X\to \R^m$ when $f$ is the MDS embedding given in Section~\ref{ss:infinite-mds}.
\end{theorem}

\begin{proof}
Let $K_B \colon X \times X \to \R$ be defined as in equation~\eqref{eq : kernel B}.
For  simplicity of notation, let $T_{B} \colon L^2(X, \mu) \to L^2(X, \mu)$ denote the Hilbert--Schmidt integral operator associated to $K_B$.
So
\[ [T_{B}](g)(x) = \int K_B(x,s)g(s) \mu(\diff s).\]
Let $\lambda_1 \geq \lambda_2 \geq \ldots$ denote the eigenvalues of $T_B$, some of which might be negative.
By the spectral theorem of compact self-adjoint operators (Theorem~\ref{thm:spectral}), the eigenfunctions $\{\phi_i\}_{i \in \N}$ of $T_B$ form an orthonormal basis of $L^2(X, \mu)$, and the operator $T_{B}$ can be expressed as
\[ T_{B}= \sum\limits_i \lambda_i \langle \phi_i , \cdot \rangle \phi_i = \sum\limits_i \lambda_i \phi _i\otimes \phi_i .\] 

Let $\{e_i\}_{i \in \N}$ be another orthonormal basis of $L^2(X, \mu)$.
Define $M_B \colon L^2(X, \mu) \to L^2(X, \mu)$ as follows
 \[ M_B  = \sum\limits_i \langle e_i, \cdot \rangle \phi_i = \sum\limits_i \phi _i\otimes e_i.\]
Indeed, $M_B$ is an orthogonal operator and it can be thought of as a ``change of basis'' operator from the $\{e_i\}$ basis to the $\{\phi_i\}$ basis.
The adjoint of $M_B$, denoted by $M'_B \colon L^2(X, \mu) \to L^2(X, \mu)$, is defined as follows
\[ M'_B = \sum\limits_i \langle \phi_i, \cdot \rangle e_i = \sum\limits_i e _i\otimes \phi_i.\]
Lastly, define the Hilbert--Schmidt operator $S_B \colon L^2(X, \mu) \to L^2(X, \mu)$ as follows
\[ S_B= \sum\limits_i \lambda_i \langle e_i, \cdot \rangle e_i =\sum\limits_i \lambda_i e _i\otimes e_i.\]
With respect to the basis $\{e_i\}$, the operator $S_B$ can be thought of as an infinite analogue of a diagonal matrix.
It can be shown that $T_B = M_B \circ S_B \circ M'_B$, consequently $ M'_B \circ T_B \circ M_B = S_B $ since $ M_B$ is an orthogonal operator.

We are attempting to minimize $\strain(f) = \tr ((T_B-T_{\hat B})^2)$ over all symmetric positive semi-definite $L^2$-kernels $K_{\hat B}$ of rank at most $m$ (for $f\colon X\to \R^m$), where we allow $m=\infty$ (for $f\colon X\to \ell^2$).
Here $T_{\hat B}$ is defined as $[T_{\hat B}](g)(x) = \int K_{\hat B} (x,s)g(s) \mu(\diff s)$.
Let ${\hat \lambda_1} \geq {\hat \lambda_2} \geq \ldots \geq 0$ denote the eigenvalues of $T_{\hat B}$ with corresponding eigenfunctions $\{ \hat \phi_i \}_{i \in \N}$, where in the case of $f\colon X\to \R^m$ we require $\hat \phi_i=0$ for $i>m$.
We have a similar factorization $T_{\hat B} = M_{\hat B} \circ S_{\hat B} \circ {M'}_{\hat B}$ for the analogously defined operators $M_{\hat B}$, $S_{ \hat B}$, and ${M'}_{\hat B}$.
For the time being, we will think of ${\hat \lambda_1} \geq {\hat \lambda_2} \geq \ldots \geq 0$ as fixed, and will optimize $T_{\hat B}$ over all possible ``change of basis'' orthogonal operators $M_{\hat B}$.

Note $ M'_B \circ T_{\hat B} \circ M_B = T_G \circ  S_{\hat B}  \circ T'_{G}$, where $T_G = M'_B \circ M_{\hat B}$ is an orthogonal operator.
Therefore, we have
\begin{align}
\strain(f)
& = \tr ((T_B-T_{\hat B})^2)\nonumber \\
& = \tr ((T_B-T_{\hat B}) M_B M'_B (T_B-T_{\hat B})) \nonumber \\
& = \tr (M'_B (T_B-T_{\hat B})(T_B-T_{\hat B})  M_B) &&\text{by Proposition~\ref{prop:trace-conj}} \nonumber \\
& = \tr (M'_B (T_B-T_{\hat B}) M_B M'_B  (T_B-T_{\hat B}) M_B)\nonumber \\
& = \tr((  S_B  - T_G  S_{\hat B} T'_G)^2) \nonumber \\ 
& = \tr( S_B ^2)- \tr( S_B T_G  S_{\hat B} T'_G) - \tr( T_G  S_{\hat B} T'_G  S_B ) + \tr( S_{\hat B}^2) \nonumber \\ 
& =  \tr( S_B ^2)- 2\tr(S_B T_G  S_{\hat B} T'_G) + \tr( S_{\hat B}^2 ) \label{eq:strain-minimization}
\end{align} 
In the above, we are allowed to apply Proposition~\ref{prop:trace-conj} because the fact that $(T_B-T_{\hat B})$ is  Hilbert--Schmidt implies that $(T_B-T_{\hat B})M_B$, and hence also $M'_B(T_B-T_{\hat B})$, are Hilbert--Schmidt.

The loss function $\strain(f)$ is minimized by choosing the orthogonal operator $T_G$ that maximizes $\tr(S_B T_G  S_{\hat B} T'_G)$.
We compute
\begin{equation}\label{eq: trace-maximization}
\tr( S_B T_G  S_{\hat B} T'_G) = \sum\limits _{i,j} \lambda_i \hat \lambda_j \langle \hat{\phi _j}, \phi_i \rangle ^2 = \sum\limits_i \lambda_i\left( \sum\limits_j \hat \lambda_j \langle \hat{\phi _j}, \phi_i \rangle ^2\right) = \sum\limits_i h_i \lambda_i,
\end{equation}
where $h_i = \sum\limits_j \hat \lambda_j\langle \hat{\phi _j}, \phi_i \rangle ^2$.
Notice that
\[ h_i\geq 0 \mbox{ and } \sum\limits_i h_i = \sum\limits_j \hat \lambda_j \sum\limits_i \langle \hat \phi_j, \phi_i \rangle ^2 = \sum\limits_j \hat \lambda_j.\]
This follows from the fact that $\hat \phi_j = \sum\limits_i \langle \hat \phi_j, \phi_i \rangle \phi_i$ and
\[\sum\limits_i \langle \hat \phi_j, \phi_i \rangle ^2 = \left\langle \hat \phi_j,  \sum\limits_i \langle \hat \phi_j, \phi_i \rangle \phi_i \right\rangle = \langle \hat \phi_j, \hat \phi_j \rangle = 1.\]

Since $\sum\limits_i h_i = \sum\limits_j\hat{\lambda}_j$ is fixed and ${\hat \lambda_1} \geq {\hat \lambda_2} \geq \ldots \geq 0$, we maximize~\eqref{eq: trace-maximization} by choosing $h_1$ as large as possible.
Indeed, maximize $h_1$ by choosing $\hat{\phi}_1=\phi_1$, so that $\langle \hat \phi_1, \phi_1 \rangle = 1$ and $\langle \hat \phi_1, \phi_i \rangle = 0$ for all $i \neq 1$.
We will show that this choice of $\hat \phi_1$ can be completed into an orthonormal basis $\{\hat \phi_j \}_{j \in \N}$ for $L^2(X)$ in order to form a well-defined and optimal positive semidefinite kernel $K_{ \bar B}$ (of rank at most $m$ in the $f\colon X\to \R^m$ case).
Next, note $\sum\limits _{i=2} h_i \lambda _i$ is maximized by choosing $h_2$ as large as possible, which is done by choosing $\hat{\phi}_2=\phi_2$.
It follows that for ${\hat \lambda_1} \geq {\hat \lambda_2} \geq \ldots \geq 0$ fixed,~\eqref{eq: trace-maximization} is maximized, and hence \eqref{eq:strain-minimization} is minimized, by choosing $\hat{\phi}_i=\phi_i$ for all $i$.
Hence,
\[T_G = \sum\limits_{i,j} \langle \hat {\phi_i}, \phi_j \rangle e_j \otimes e_i =\sum\limits_i e _i\otimes e_i.\]
Therefore, we can do no better than choosing $T_G$ to be the identity operator, giving $M_{\hat B} = M_B$.

We will now show how to choose the eigenvalues ${\hat \lambda_1} \geq {\hat \lambda_2} \geq \ldots \geq 0$.
As in Lemmas~\ref{lem: closest-nonneg-l2} and~\ref{lem: closest-nonneg-Rm}, for $f\colon X\to\ell^2$ we define
\[ \bar \lambda_i =
\begin{cases} 
\lambda_i & \text{if }\lambda_i \geq 0 \\
0 & \text{if }\lambda_i < 0,
\end{cases}\]
and for $f\colon X\to\R^m$ we define
\[ \bar \lambda_i =
\begin{cases} 
\lambda_i & \text{if }\lambda_i \geq 0\text{ and }i\le m \\
0 & \text{if }\lambda_i < 0\text{ or }i>m.
\end{cases}\]
Define $K_{\bar B}(x,s) = \sum\limits_{i=1} ^ \infty \bar \lambda_i \phi _i (x) \phi_i (s)$, where each eigenfunction $\phi _i$ of $K_{\bar B}$ is the eigenfunction $\phi _i$ of $K_B$ corresponding to the eigenvalue $\bar \lambda _i$.

We compute, that over all possible choices of eigenvalues, the optimal $\strain$ is given by the choices $\bar{\lambda}_i$:
\begin{align*}
\strain(f) &= {\tr}((T_B-T_{\hat B})^2)\\ 
&= \tr((  S_B  - T_G  S_{\hat B} T'_G)^2) \\
& \geq \tr((  S_B  -  S_{\hat B} )^2) \\
& = \| \lambda - \hat \lambda\|_2^2\\
& \geq \| \lambda - \bar \lambda\|_2^2 &&\text{by Lemma~\ref{lem: closest-nonneg-l2} or~\ref{lem: closest-nonneg-Rm}}\\
& = \tr((  S_B  - S_{\bar B} )^2) \\ 
& = \tr(( M_B (  S_B  - S_ {\bar B})M'_B )^2) &&\text{by Proposition~\ref{prop:trace-conj}}  \\
& = {\tr}((T_B-T_{\bar B})^2).
\end{align*}
Therefore, the loss function $\strain(f)$ is minimized when $f$ and $T_{\hat B}$ are defined via the MDS embedding in Section~\ref{ss:infinite-mds}.
\end{proof}

The following table shows a comparison of various elements of classical MDS and infinite MDS (as described in section~\ref{ss:infinite-mds}). 
Our table is constructed analogously to a table on the Wikipedia page~\cite{wiki:FPCA} that shows a comparison of various elements of Principal Component Analysis (PCA) and Functional Principal Component Analysis (FPCA). 
In Chapter~\ref{chap: convergence}, we address convergence questions for MDS more generally.

\begin{sidewayspage}

\begin{table}[h]
    \caption{A comparison table of various elements of classical MDS and infinite MDS.}
    \label{table:sample}
    \begin{center}
\renewcommand{\arraystretch}{2.7}
\begin{tabular}{ |p{5cm}||p{7.5cm}|p{8.5cm}|  }
 \hline
 \textbf{Elements} &  \textbf{Classical MDS}  &  \textbf{Infinite MDS}\\
 \hline
Data  & $(X_n,d)$    & $(X, d_X, \mu)$ \\
 \hline
Distance Representation &   $D_{i,j} = d(x_i, x_j),\quad D \in \mathcal M _{n \times n}$ & $K_D(x,s) = d_X(x,s) \in L^2_{\mu \otimes \mu }(X\times X)$  \\
 \hline
Linear Operator &$B = -\frac{1}{2}HD^{(2)}H$   & $[T_{K_B}\phi](x) = \int \limits_X K_B(x,s) \phi(s) \mu(\diff s)$ \\
  \hline
 Eigenvalues & $\lambda_1 \geq \lambda_2 \geq \ldots \geq \lambda_n $ & $\hat \lambda_1 \geq \hat \lambda_2 \geq \ldots $\\
  \hline
Eigenvectors/Eigenfunctions &$v^{(1)}, v^{(2)}, \ldots, v^{(m)} \in \R^n$ & $\phi_1(x), \phi_2(x), \ldots \in L^2(X) $ \\
 \hline
Embedding in $\R^m$ or $\ell^2$& $  f(x_i) = \left(\sqrt{ \lambda_1} v^{(i)}_1,\sqrt{ \lambda_2} v^{(i)}_2, \ldots, \sqrt{ \lambda_m} v^{(i)}_m \right)$  &  $f(x) = \left(\sqrt{\hat \lambda_1} \phi_1(x),\sqrt{\hat \lambda_2} \phi_2(x),\sqrt{\hat \lambda_3} \phi_3(x),\ldots\right)$  \\
 \hline
 Strain Minimization& $\sum\limits_{i,j=1}^n(b_{i,j}-\hat{b}_{i,j})^2$  &$\int \int \left( K_B(x,t) - K_{\hat B}(x,t) \right)^2 \mu(\diff t) \mu(\diff x)$ \\
 \hline
\end{tabular}
    \end{center}
\end{table}
\end{sidewayspage}

\renewcommand{\arraystretch}{1}

\chapter{MDS of the Circle}
\label{chap: MDS circle}

In this chapter, we consider the MDS embeddings of the circle equipped with the (non-Euclidean) geodesic metric.
The material in this section is closely related to~\cite{von1941fourier}, even though~\cite{von1941fourier} was written prior to the invention of MDS.
By using the known eigendecomposition of circulant matrices, we are able to give an explicit description of the MDS embeddings of evenly-spaced points from the circle.
This is a motivating example for the convergence properties studied in Chapter~\ref{chap: convergence}, which will show that the MDS embeddings of more and more evenly-spaced points will converge to the MDS embedding of the entire geodesic circle.
We also remark that the geodesic circle is a natural example of a metric space whose MDS embedding in $\ell^2$ is better (in the sense of $\strain$-minimization) than its MDS embedding into $\R^m$ for any finite $m$.

We describe the eigenvalues and eigenvectors of circulant matrices in Section~\ref{sec: circulant}, and use this to describe the MDS embeddings of evenly-spaced points from the circle in Section~\ref{sec: circle}.
In Section~\ref{sec: vonNeumann}, we describe the relationship between the MDS embedding of the circle and the much earlier work of~\cite{von1941fourier,wilson1935certain}.
We also refer the reader to the conclusion (Chapter~\ref{chap: conclusion}) for open questions on the MDS embeddings of geodesic $n$-spheres $S^n$ into $\mathbb{R}^m$.

\section{Background on Circulant Matrices}\label{sec: circulant}

Let $\bB$ be an $n\times n$ matrix.
The matrix $\bB$ is \emph{circulant} if each row is a cyclic permutation of the first row, in the form as shown below.
\[ \bB = \begin{pmatrix}
b_{0} & b_{1} & b_{2} & \ldots & b_{n-3} & b_{n-2} & b_{n-1} \\
b_{n-1} & b_{0} & b_{1} & \ldots & b_{n-4} & b_{n-3} & b_{n-2} \\
b_{n-2} & b_{n-1} & b_{0} & \ldots & b_{n-5} & b_{n-4} & b_{n-3} \\
\vdots & \vdots & \vdots & & \vdots & \vdots & \vdots \\
b_{3} & b_{4} & b_{5} & \ldots & b_{0} & b_{1} & b_{2} \\
b_{2} & b_{3} & b_{4} & \ldots & b_{n-1} & b_{0} & b_{1} \\
b_{1} & b_{2} & b_{3} & \ldots & b_{n-2} & b_{n-1} & b_{0}
\end{pmatrix}
\]
The first row of a circulant matrix determines the rest of the matrix.
Furthermore, a circulant matrix $\bB$ has a basis of eigenvectors of the form $x_k(n) = \begin{pmatrix} w_{n}^{0k} & w_{n}^{1k} & \ldots & w_{n}^{(n-1)k} \end{pmatrix}^\top$ for $0\le k\le n-1$, where $ w_{n} = e^{\frac{2\pi i}{n}}$ and $\top$ denotes the transpose of a matrix.
The eigenvalue corresponding to $x_k(n)$ is
\begin{equation} \lambda_k(n) = \sum\limits\limits_{j=0}^{n-1} b_{j}w_{n}^{jk} = b_{0}w_{n}^{0k} + b_{1}w_{n}^{1k}+ \cdots + b_{n-1}w_{n}^{(n-1)k}.
\label{eq:eigenvalues}
\end{equation}

If the circulant matrix $\bB$ is also symmetric, then $b_i = b_{n-i}$ for $1 \leq i \leq n-1$.
Thus, $\bB$ has the form
\begin{equation}\label{eq:symmetric-circulant}
\bB =
\begin{pmatrix}
b_{0} & b_{1} & b_{2} & \ldots & b_{3} & b_{2} & b_{1} \\
b_{1} & b_{0} & b_{1} & \ldots & b_{4} & b_{3} & b_{2} \\
b_{2} & b_{1} & b_{0} & \ldots & b_{5} & b_{4} & b_{3} \\
\vdots & \vdots & \vdots & & \vdots & \vdots & \vdots \\
b_{3} & b_{4} & b_{5} & \ldots & b_{0} & b_{1} & b_{2} \\
b_{2} & b_{3} & b_{4} & \ldots & b_{1} & b_{0} & b_{1} \\
b_{1} & b_{2} & b_{3} & \ldots & b_{2} & b_{1} & b_{0}
\end{pmatrix}.
\end{equation}
The eigenvalues $\lambda_k(n)$ are then real and of the form
\begin{align*}
\lambda_{k} &= b_{0} + 2b_{1}\Re(w_{n}^{1k})+ \cdots + 2b_{\frac{n-1}{2}}\Re(w_{n}^{\frac{(n-1)}{2}k})&&\mbox{if }n\mbox{ is odd, and}\\
\lambda_{k} &= b_{0} + b_{{\frac{n}{2}}}w_{n}^{\frac{n}{2}k} + 2b_{1}\Re(w_{n}^{1k})+ \cdots + 2b_{\frac{n}{2}-1}\Re(w_{n}^{(\frac{n}{2}-1)k})&&\mbox{if }n\mbox{ is even.}     
\end{align*}

\section{MDS of Evenly-Spaced Points on the Circle}\label{sec: circle}

Let $S^1$ be the unit circle (i.e.\ with circumference $2\pi$), equipped with the geodesic metric $d$ which can be simply thought of as the shortest path between two given points in a curved space (in this case, the circle).
We let $S^1_n$ denote a set of $n$ evenly spaced points on $S^1$.
In Proposition~\ref{prop:S1n}, we show that the MDS embedding of $S^1_n$ in $\R^m$ lies, up to a rigid motion, on the curve $\gamma_n\colon S^1\to\R^m$ defined by
\[ \gamma_m(\theta) = (a_1(n)\cos(\theta),a_1(n)\sin(\theta),a_3(n)\cos(3\theta),a_3(n)\sin(3\theta),a_5(n)\cos(5\theta),a_5(n)\sin(5\theta),\ldots)\in\R^m, \]
where $\lim_{n\to\infty} a_j (n)= \frac{\sqrt{2}}{j}$ (with $j$ odd).

Let $\bD$ be the distance matrix of $S^1_n$, which is determined (up to symmetries of the circle) by
\[
\frac{n}{2\pi} d_{0j} =
\begin{cases}
j & \text{for } 0 \leq j \leq \lfloor\frac{n}{2}\rfloor \\
n - j & \text{for }  \lceil\frac{n}{2}\rceil \leq j \leq n-1,
\end{cases}
\]
where $\lfloor . \rfloor$ denotes the floor function and $\lceil . \rceil$ denotes the ceiling function.
Let $\bA = (a_{ij})$ with $a_{ij} = -\frac{1}{2}d^2_{ij}$ and let $\bB = \bH\bA\bH$, where $\bH = \bI - n^{-1}\mathbf{1}\mathbf{1}^\top$ is the centering matrix. 
The distance matrix $\bD$ is real-symmetric circulant, and it follows that $\bB$ is real-symmetric circulant with its form as shown in equation~\eqref{eq:symmetric-circulant}.
After applying symmetries of the circle, the entries of the first row vector $(b_0, b_1, \cdots, b_{n-1})$ of $\bB$ can be written explicitly as 
\[b_{j} = -\frac{1}{2}\left(d^2_{0j} - \frac{1}{n}\sum\limits\limits_{k=0}^{n-1} d^2_{0k} - \frac{1}{n}\sum\limits\limits_{k=0}^{n-1} d^2_{kj} + \frac{1}{n^2}\sum\limits\limits_{k=0}^{n-1} \sum\limits\limits_{l=0}^{n-1} d^2_{kl}\right) = -\frac{1}{2}(d^2_{0j} - c_n),\]
where a formula for $c_n=\frac{1}{n} \sum\limits_{k=0}^{n-1} d^2_{0k}$ can be computed explicitly\footnote{If $n$ is odd we have $c_n = \frac{2}{n}\sum\limits_{k=0}^{\frac{n-1}{2}} \left(\frac{2\pi k}{n}\right)^2 =\frac{\pi^2}{3n^2}(n^2-1)$,
and if $n$ is even we have
\[c_n = \frac{1}{n}\left(d_{0,\frac{n}{2}}^2 + 2\sum\limits\limits_{k=0}^{\frac{n}{2}-1} d_{0k}^2\right) = \frac{1}{n}\left(\pi^2 + 2\sum\limits\limits_{k=0}^{\frac{n}{2}-1} \left(\frac{2\pi k}{n}\right)^2\right) = \frac{\pi^2}{3n^2}(n-1)(n-2) + \frac{\pi ^2}{n}.\]
}.
Furthermore, let $\lambda_{k}(n)$ denote the $k$th eigenvalue of the matrix $\bB$ corresponding to the $k$th eigenvector $ x_k(n)$.
A basis of eigenvectors for $\bB$ consists of $ x_k(n) =\begin{pmatrix} w_{n}^{0k} & w_{n}^{1k} & \ldots & w_{n}^{(n-1)k} \end{pmatrix}^\top$ for $0\le k\le n-1$, where $ w_{n} = e^{\frac{i 2\pi}{n}}$.

\begin{lemma}\label{lem:eigenvalue_S1n}
We have $\lambda_{k}(n) = -\frac{1}{2}\sum\limits\limits_{j=0}^{n-1} d_{0j}^2w_{n}^{jk}$ for $0\le k\le n-1$.
\end{lemma}

\begin{proof}
We compute
\[\lambda_{k}(n) = \sum\limits\limits_{j=0}^{n-1} b_{j}w_{n}^{jk}= \sum\limits\limits_{j=0}^{n-1}-\frac{1}{2}(d^2_{0j} - c_n)w_{n}^{jk} = -\frac{1}{2}\sum\limits\limits_{j=0}^{n-1}d_{0j}^2 w_{n}^{jk} +\frac{1}{2} c_n\sum\limits\limits_{j=0}^{n-1} w_{n}^{jk},\]
and the property follows by noting that $\sum\limits_{j=0}^{n-1} w_{n}^{jk}=0$.
\end{proof}

\begin{corollary}\label{prop:S1n-eigenvalues}
The $k$th eigenvalue $\lambda_k(n)$ corresponding to the eigenvector $x_k(n)$ satisfies
\[\lambda_0(n)=0,\quad\mbox{and}\quad\lim_{n\to\infty} \frac{\lambda_{k}(n)}{n} = \frac{(-1)^{k+1}}{k^2}\quad\mbox{for}\quad k\ge 1.\]
\end{corollary}

\begin{proof}
If $k=0$, we have $x_0(n) = \mathbf{1}$ with eigenvalue $\lambda_{0}(n) = 0$ since $\bB$ is a double-centered matrix.
Hence we restrict attention to $1\le k\le n-1$.
By Lemma~\ref{lem:eigenvalue_S1n}, we have we have
\[
\frac{\lambda_{k}(n)}{n} 
= -\frac{1}{2n} \sum\limits_{j=0}^{n-1} d_{0j}^2w_{n}^{jk} 
= -\frac{1}{2n}\left(\frac{2\pi}{n}\right)^2\sum\limits_{j
= -\lfloor\frac{n}{2}\rfloor}^{\lfloor\frac{n-1}{2}\rfloor} j^2e^{(\frac{2\pi}{n}jk)i} 
= -2\pi^2\left(\frac{1}{n}\right)\sum\limits_{j
= -\lfloor\frac{n}{2}\rfloor}^{\lfloor\frac{n-1}{2}\rfloor} \left(\frac{j}{n}\right)^2 e^{(2\pi (\frac{j}{n})k)i} 
=: S_n.\]
Since $S_n$ is the left-hand Riemann sum (with $n$ subintervals) of the below integral, we use integration by parts to get
\begin{equation}\label{eq: limit evals}
\lim_{n\to\infty} \frac{\lambda_{k}(n)}{n} = \lim_{n\to\infty} S_n = -2\pi^2\int_{-\frac{1}{2}}^{\frac{1}{2}} x^2e^{2\pi xki}\ \diff x = \frac{(-1)^{k+1}}{k^2}.
\end{equation}
\end{proof}

\begin{lemma} \label{lem:odd-eigen}
For all $k$ odd, there exists $N \in \N$ sufficiently large (possibly depending on $p$ odd) such that for all $n\ge N$, the eigenvalues $\lambda_k(n)$ satisfy the following property 
\[\lambda_1(n) \geq \lambda_3(n) \geq\lambda_5(n) \geq \cdots \geq \lambda_p(n) \geq 0,\]
where $\lambda_k(n)$ is the eigenvalue corresponding to $x_k(n)$.
\end{lemma}

\begin{proof}
From equation~\eqref{eq: limit evals}, for $k$ odd we have,
\[\lim_{n\to\infty} \frac{\lambda_{k}(n)}{n} = \frac{1}{k^2}.\]
Therefore, for each $k$ odd, $\exists N_p \in \N$ such that $\forall n \geq N_p$, we have 
\begin{equation}\lambda_1(n) \geq \lambda_3(n) \geq\lambda_5(n) \geq \cdots \geq \lambda_p(n) \geq 0. \end{equation}
\end{proof}

\begin{remark}
We conjecture that Lemma~\ref{lem:odd-eigen} is true for $N_p=1$.
\end{remark}

For a real-symmetric circulant matrix, the real and imaginary parts of the eigenvectors $x_k(n)$ are eigenvectors with the same eigenvalue.
These eigenvectors correspond to a discrete cosine transform and a discrete sine transform.
Let $u_k(n)$ and $v_k(n)$ denote the real and imaginary parts of $x_k(n)$ respectively.
In general, \[ u_k(n) = \begin{pmatrix} 1 & \cos\theta & \cos 2 \theta & \ldots & \cos(n-1) \theta \end{pmatrix}^\top, \] and \[v_k(n) = \begin{pmatrix}0 & \sin \theta & \sin 2 \theta & \ldots & \sin (n-1) \theta \end{pmatrix}^\top,\] where $\theta =\frac{2\pi k}{n}$.

Since $\bB$ is a real symmetric matrix, its orthogonal eigenvectors can also be chosen real.
We have $\Re (w_{n}^{k})= \Re (w_{n}^{n-k})$ and $\Im (w_{n}^{k})= - \Im (w_{n}^{n-k})$.
A new basis of $n$ eigenvectors in $\R^n$ of $\bB$ can be formed from $u_k(n)$ and $v_k(n)$ as follows.
Let $\mathcal E =\{ u_k(n), v_k(n)~|~ k = 1,2, \cdots, \frac {n}{2}-1\}$ and $\mathcal O = \{ u_k(n), v_k(n)~|~  k = 1,2, \cdots, \frac{n-1}{2}\}$.
If $n$ is even, a set of $n$ linearly independent eigenvectors is 
\[\mathcal B_e =\mathcal E  \cup \{u_0(n),u_\frac{n}{2}(n)\}.\]
Furthermore, if $n$ is odd, a set of $n$ linearly independent eigenvectors is 
\[\cB_o = \mathcal O \cup \{u_0(n) \}.\]
In each of the sets $\mathcal E$ and $\mathcal O$, the eigenvalue of $u_k(n)$ is the same as the eigenvalue of $v_k(n)$ for each $k$.
So, the eigenvalues in $\mathcal E$ and $\mathcal O$ come in pairs.
The only eigenvalues that don't come in pairs correspond to eigenvectors $x_k(n)$ that are purely real, namely $u_0(n)$, and $u_\frac{n}{2}(n)$ (for $n$ even).
Indeed, the eigenvalue corresponding to $u_0(n)$ is $\lambda_0 (n) = 0$.
Furthermore, the new eigenvectors are real and mutually orthogonal.

\begin{lemma} \label{prop:L2-norm}
The $\ell^2$-norm of $u_k(n)$ and $v_k(n)$ satisfy the following property 
\begin{equation}
\lim_{n\to\infty} {\frac{\|u_k(n)\|}{\sqrt n}} = \lim_{n\to\infty} {\frac{\|v_k(n)\|}{\sqrt n}}= {\frac{1}{\sqrt 2}}.
\end{equation}
\end{lemma}

\begin{proof} Consider the following where $\theta =\frac{2\pi k}{n}$, and where we have used the fact that a limit of Riemann sums converges to the corresponding Riemann integral.
\[\lim_{n\to\infty} \frac{\|u_k(n)\|^2}{n} = \lim_{n\to\infty} \frac{1}{n} \sum\limits _{j=0}^{n-1} \cos^2(j\theta) = \lim_{n\to\infty}\frac{1}{n}\sum\limits _{j=0}^{n-1}\cos^2(2\pi k (\tfrac{j}{n})) = \int\limits_0^1 \cos^2(2\pi k x)\diff x = \frac{1}{2}, \mbox{ and }\] 
\[ \lim_{n\to\infty} \frac{\|v_k(n)\|^2}{n} = \lim_{n\to\infty} \frac{1}{n} \sum\limits _{j=0}^{n-1} \sin^2(j\theta) = \lim_{n\to\infty}\frac{1}{n}\sum\limits _{j=0}^{n-1}\sin^2(2\pi k (\tfrac{j}{n})) = \int\limits_0^1 \sin ^2(2\pi k x)\diff x = \frac{1}{2}.\]
The definite integrals can be computed by hand using double angle formulas. 
Therefore, the result follows.
\end{proof}

Let $\bLambda_m$ be the $m\times m$ diagonal matrix of the largest $m$ eigenvalues of $\bB$ sorted in descending order, and let $\bGamma_m$ be the $n\times m$ matrix of the corresponding $m$ eigenvectors.
From Lemma~\ref{lem:odd-eigen}, $\bLambda_m$ can be written as follows,
\[
\bLambda_m  =
\begin{pmatrix}
\lambda_1 (n) & & & & &\\
& \lambda_1 (n) & & & &\\
& & \lambda_3 (n) & & &\\
& & & \lambda_3 (n) & &\\
& & & &  \ddots & \\
& & & & & \lambda_{m}(n)
\end{pmatrix}
\] 
and $\bGamma_m$ can be constructed in different of ways.
We'll construct $\bGamma_m$ as follows,
\[\bGamma_m  = \begin{pmatrix} \frac{u_1(n)}{\|u_1(n)\|}& \frac{v_1(n)}{\|v_1(n)\|}& \frac{u_3(n)}{\|u_3(n)\|}& \frac{v_3(n)}{\|v_3(n)\|}& \cdots & \frac{v_m(n)}{\|v_m(n)\|}\end{pmatrix}.\] 
The classical MDS embedding of $S^1_n$ consists of the $n$ points in $\R^m$ whose coordinates are given by the $n$ rows of the $n\times m$ matrix $\bX=\bGamma_m\bLambda_m^{1/2}$.

\begin{proposition}\label{prop:S1n}
The classical MDS embedding of $S^1_n$ lies, up to a rigid motion of $\R^m$, on the curve $\gamma_m\colon S^1\to\R^m$ defined by
\[ \gamma_m(\theta) = (a_1(n)\cos(\theta),a_1(n)\sin(\theta),a_3(n)\cos(3\theta),a_3(n)\sin(3\theta),a_5(n)\cos(5\theta),\ldots)\in\R^m, \]
where $\lim_{n\to\infty} a_j (n)= \frac{\sqrt{2}}{j}$ (with $j$ odd).
\end{proposition}

\begin{proof}	
The coordinates of the points of the MDS embedding of $S^1_n$ are given by the $n\times m$ matrix $\bX=\bGamma_m\bLambda_m^{1/2}$.
This implies the coordinates of the $n$ configuration points in $\R^m$ are given by \[(a_1(n)\cos(\theta),b_1(n)\sin(\theta),a_3(n)\cos(3\theta),b_3(n)\sin(3\theta),a_5(n)\cos(5\theta),b_5(n)\sin(5\theta),\ldots)\in\R^m,\]
where $\theta = \frac{2\pi k}{n}$ and $0\leq k \leq n-1$ and (for $j$ odd)
\[a_j (n)= \frac{\sqrt{\lambda_j(n)}}{\|u_j(n)\|} \quad \mbox{and} \quad b_j (n) = \frac{\sqrt{\lambda_j(n)}}{\|v_j(n)\|}.\]
From Corollary~\ref{prop:S1n-eigenvalues} and Lemma~\ref{prop:L2-norm}, we have
\[\lim_{n\to\infty} a_j(n) 
= \lim_{n\to\infty} \frac{\sqrt{\lambda_j(n)}}{\|u_j(n)\|}
=\lim_{n\to\infty} \frac{\frac{\sqrt{\lambda_j(n)}}{\sqrt n}}{\frac{\|u_j(n)\|}{\sqrt n}} 
= \frac{\sqrt 2}{j},\]
and similarly $\lim_{n\to\infty} b_j(n) = \frac{\sqrt 2}{j}$.
Therefore, we can say that the MDS embedding of $S^1_n$ lies, up to a rigid motion of $\R^m$, on the curve $\gamma_m\colon S^1\to\R^m$ defined by
\[ \gamma_m(\theta) = (a_1(n)\cos(\theta),a_1(n)\sin(\theta),a_3(n)\cos(3\theta),a_3(n)\sin(3\theta),a_5(n)\cos(5\theta),\ldots)\in\R^m, \]
where $\lim_{n\to\infty} a_j (n)= \frac{\sqrt{2}}{j}$ (with $j$ odd).
\end{proof}
 
\begin{figure}[h]
\centering  \includegraphics[width=0.7\textwidth]{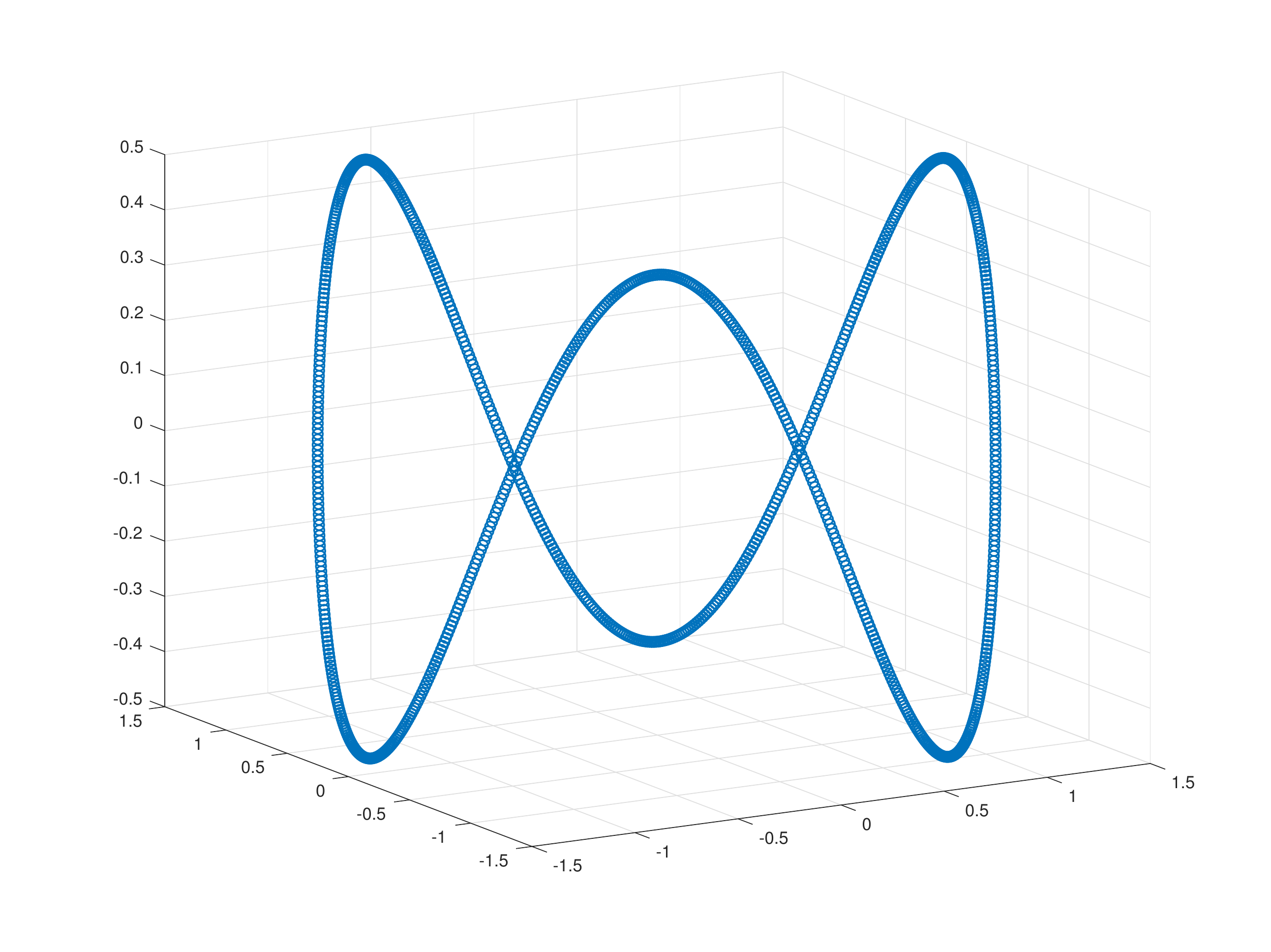}
\caption{MDS embedding of $S^1_{1000}$.}
\label{fig: MDS of Circle}
\end{figure}

Indeed, Figure~\ref{fig: MDS of Circle} shows the MDS configuration in $\R ^3$ of 1000 points on $S^1$ obtained using the three largest positive eigenvalues.

\section{Relation to Work of von Neumann and Schoenberg}\label{sec: vonNeumann}

The MDS embeddings of the geodesic circle are closely related to~\cite{von1941fourier}, which was written prior to the invention of MDS.
In~\cite[Theorem~1]{von1941fourier}, von Neumann and Schoenberg describe (roughly speaking) which metrics on the circle one can isometrically embed into the Hilbert space $\ell^2$.
The geodesic metric on the circle is not one of these metrics.
However, the MDS embedding of the geodesic circle into $\ell^2$ must produce a metric on $S^1$ which is of the form described in~\cite[Theorem~1]{von1941fourier}.
See also~\cite[Section~5]{wilson1935certain} and~\cite{blumenthal1970theory,bogomolny2003spectral,dattorro2010convex}.

\chapter{Convergence of MDS}
\label{chap: convergence}

We saw in the prior chapter how sampling more and more evenly-spaced points from the geodesic circle allowed one to get a sense of how MDS behaves on the entire circle. 
In this chapter, we address convergence questions for MDS more generally.
Convergence is well-understood when each metric space has the same finite number of points~\cite{sibson1979studies}, but we are also interested in convergence when the number of points varies and is possibly infinite.

This chapter is organized as follows.
In Section~\ref{sec: robustness}, we survey Sibson's perturbation analysis~\cite{sibson1979studies} for MDS on a fixed number of $n$ points. 
Next, in Section~\ref{sec: convergence random}, we survey results of~\cite{bengio2004learning,koltchinskii2000random} on the convergence of MDS when $n$ points $\{x_1,\ldots,x_n\}$ are sampled from a metric space according to a probability measure $\mu$, in the limit as $n\to\infty$.
Unsurprisingly, these results rely on the law of large numbers.
In Section~\ref{sec: convergence deterministic}, we reprove these results under the (simpler) deterministic setting when points are not randomly chosen, and instead we assume that the corresponding finite measures $\mu_n = \frac{1}{n}\sum\limits_{i=1}^{n} \delta_{x_i}$ (determined by $n$ points) converge to $\mu$.
This allows us, in Section~\ref{sec: convergence deterministic arbitrary}, to consider the more general setting where we have convergence of \emph{arbitrary} probability measures $\mu_n\to\mu$.
For example, in what sense do we still have convergence of MDS when each measure $\mu_n$ in the converging sequence $\mu_n\to\mu$ has infinite support?
Finally, in Section~\ref{sec: convergence GW}, we ask about the even more general setting where we have the convergence of arbitrary metric measure spaces $(X_n, d_n, \mu_n)\to(X, d, \mu)$ in the Gromov--Wasserstein distance.

\section{Robustness of MDS with Respect to Perturbations}\label{sec: robustness}

In a series of papers~\cite{sibson1978studies,sibson1979studies,sibson1981studies}, Sibson and his collaborators consider the robustness of multidimensional scaling with respect to perturbations of the underlying distance or dissimilarity matrix as illustrated in Figure~\ref{fig:measure_Sibson}.
In particular,~\cite{sibson1979studies} gives quantitative control over the perturbation of the eigenvalues and vectors determining an MDS embedding in terms of the perturbations of the dissimilarities.
These results build upon the fact that if $\lambda$ and $v$ are a (simple, i.e.\ non-repeated) eigenvalue and eigenvector of an $n\times n$ matrix $B$, then one can control the change in $\lambda$ and $v$ upon a small symmetric perturbation of the entries in $B$.

\begin{figure}[h]
   \includegraphics[width=0.35\textwidth]{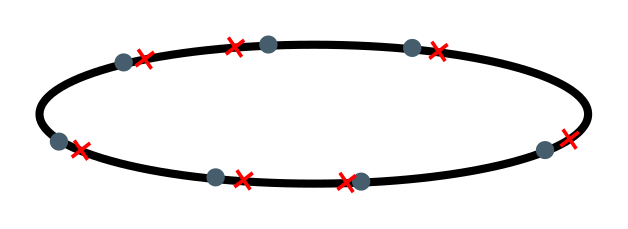}
    
    \caption{Perturbation of the given dissimilarities.}
    \label{fig:measure_Sibson}

\end{figure}

Sibson's perturbation analysis shows that if one is has a converging sequence of $n\times n$ dissimilarity matrices, then the corresponding MDS embeddings of $n$ points into Euclidean space also converge.
In the following sections, we consider the  convergence of MDS when the number of points is not fixed. Indeed, we consider the convergence of MDS when the number of points is finite but tending to infinity, and alternatively also when the number of points is infinite at each stage in a converging sequence of metric measure spaces.

\section{Convergence of MDS by the Law of Large Numbers}\label{sec: convergence random}

In this section, we survey results of~\cite{bengio2004learning,koltchinskii2000random} on the convergence of MDS when more and more points are sampled from a metric space.

Suppose we are given the data set $X_n=\{x_1,
\ldots, x_n\}$ with $x_i \in \R^k$ sampled independent and identically distributed (i.i.d.) from an unknown probability measure $\mu$ on $X$.
Define $D_{ij} = d(x_i, x_j)$ and the corresponding data-dependent kernel $K_n$ as follows
\begin{equation}\label{eq:kernel-Xn}
K_n(x,y)=-\frac{1}{2}\left(d^2(x,y)-\frac{1}{n}\sum_{i=1}^n d^2(x_i,y) -\frac{1}{n}\sum_{i=1}^n d^2(x,x_i) +\frac{1}{n^2} \sum_{i,j=1}^n d^2(x_i,x_j)\right).
\end{equation}
Define the Gram matrix $ M = -\frac{1}{2}H D^{(2)}H$, where $H = I - n^{-1}\mathbf{1}\mathbf{1}^{\top}$, and note that $M_{ij}= K_n(x_i, x_j)$ for $i,j = 1, \ldots, n$.
Assume that the (possibly data-dependent) kernel $K_n$ is bounded (i.e.\ $|K_n(x,y)| < c$ for all $x,y$ in $\R^k$).
We will assume that $K_n$ converges uniformly in its arguments and in probability to its limit $K$ as $ n \to \infty$. This means that for all $\epsilon >0$, 
\[\lim_{n \to \infty} P( \sup_{x,y \in \R^k} |K_n(x,y - K(x,y)| \geq \epsilon) =0. \]
Furthermore, assume $K_n$ is an $L^2$-kernel.
Then associated to $K_n$ is an operator 
$T_{K_n} \colon L^2(X) \to L^2(X)$ defined as
\[[T_{K_n} f](x) = \frac{1}{n}\sum\limits_{i=1}^n K_n(x, x_i)f(x_i). \]
Define $T_K \colon L^2(X) \to L^2(X)$ as
\[[T_K f](x) = \int K(x, s)f(s)\mu(\diff s), \]
where
\[ K(x,y)= \frac{1}{2} \left(-d(x,y)^2 + \int \limits_X  d(w,y)^2 \mu(\diff w) + \int \limits_X d(x,z)^2 \mu(\diff z) - \int_{X \times X} d(w,z)^2 \mu(\diff w \times \diff z) \right)\]
is defined as in Section~\ref{ss:infinite-mds}.
Therefore, we obtain the following eigensystems, 
\[ T_K f_k = \lambda _k f_k \quad \mbox{ and } \quad 
T_{K_n} \phi_{k,n} = \lambda_{k,n} \phi_{k,n}, \] where $(\lambda_k, \phi_k)$ and $(\lambda_{k,n}, \phi_{k,n})$ are the corresponding eigenvalues and eigenfunctions of $T_K$ and $T_{K_n}$ respectively.
Furthermore, when we evaluate $T_{K_n}$ at the data points $x_i\in X_n$, we obtain the following eigensystem for $M$,
\[ M v_k = \ell_k v_k,\]
where $(\ell_k, v_k)$ are the corresponding eigenvalues and eigenvectors.

\begin{lemma}\label{lem: evals-1}~\cite[Proposition 1]{bengio2004learning}
$T_{K_n}$ has in its image $m \leq n$ eigenfunctions of the form \[\phi_{k, n}(x) = \frac{\sqrt n}{\ell_k} \sum\limits_{i=1}^n v^{(i)}_kK_n(x,x_i)\] with corresponding eigenvalues $\lambda_{k,n} = \frac{\ell_k}{n}$ where $v^{(i)}_k$ denotes the $i$th entry of the $k$th eigenvector of $M$ associated with the eigenvalue $\ell_k$.
For $x_i \in X_n$, these functions coincide with the corresponding eigenvectors,    $\phi_{k, n}(x_i) = \sqrt n v^{(i)}_k .$
\end{lemma}

Indeed, $M$ has $n$ eigenvalues whereas $T_{K_n}$ has infinitely many eigenvalues, this means that $T_{K_n}$ has at most $n$ nonzero eigenvalues, and that $0$ is an eigenvalue of $T_{K_n}$ with infinite multiplicity.
In order to compare the finitely many eigenvalues of $M$ with the infinite sequence of eigenvalues of $T_{K_n}$, some procedure has to be constructed~\cite{koltchinskii2000random}.

Suppose that the eigenvalues are all non-negative and sorted in non-increasing order, repeated according to their multiplicity.
Thus, we obtain eigenvalue tuples and sequences
\[\lambda (M)= (l_1, \ldots , l_n),\] where $l_1 \geq \ldots \geq l_n$, and
 \[\lambda (T_{K_n}) = (\lambda_1, \lambda_2, \ldots)\] where $\lambda_1 \geq \lambda_2 \geq \ldots$.

To compare the eigenvalues, first embed $\lambda (M)$ into $\ell ^1$ by padding the length-$n$ vector with zeroes, obtaining
\[\lambda (M) = (l_1, \ldots , l_n, 0,0, \ldots). \]
\begin{definition}
The \emph{$\ell^2$-rearrangement distance} between (countably) infinite sequences $x$ and $y$ is defined as
\[\delta_2(x, y) = \inf_{\pi \in \mathcal G(\N)} \sum\limits_{i=1}^{\infty} (x_i - y_{\pi(i)})^2 \]
where $\mathcal G(\N)$ is the set of all bijections on $\N$.
\end{definition}

In this section and the following sections, the eigenvalues are always ordered in non-increasing order by the spectral theorem of self-adjoint operators.
We note that the $\ell^2$-rearrangement distance is simply the $\ell^2$-distance when the eigenvalues are ordered.

\begin{theorem}~\cite[Theorem~3.1]{koltchinskii2000random}\label{thm:koltchinskii}
The ordered spectrum of $T_{K_n}$ converges to the ordered spectrum of $T_{K}$ as $n \to \infty$ with respect to the $\ell^2$-distance, namely
\[\ell^2(\lambda(T_{K_n}), \lambda(T_K)) \to 0 \quad \mbox{a.s.} \] 
\end{theorem}

The theorem stated above is in fact only one part, namely equation (3.13), in the proof of ~\cite[Theorem~3.1]{koltchinskii2000random}.
We also remark that~\cite{koltchinskii2000random} uses fairly different notation for the various operators than what we have used here.

\begin{theorem}~\cite[Proposition 2]{bengio2004learning}
\label{thm: bengioConvergence}
If $K_n$ converges uniformly in its arguments and in probability, with the eigendecomposition of the Gram matrix converging, and if the eigenfunctions $\phi_{k, n}(x)$ of $T_{K_n}$ associated with non-zero eigenvalues converge uniformly in probability, then their limit are the corresponding eigenfunctions of $T_K$.
\end{theorem}

Under the given hypotheses, we are able to formulate a specific set of eigenvalues and eigenfunctions $T_K$.
As illustrated below, the choice of the eigenfunctions $\phi_{k,n}$ of $T_{K_n}$ was made to extend the finite MDS embedding to the infinite MDS embedding described in Section~\ref{ss:infinite-mds}.
However, there are other possible choices of eigenfunctions $\phi_{k,n}$. 

Consider the infinite MDS map $f_m\colon X\to \R^m$ defined in Section~\ref{ss:infinite-mds} as
\[f_m(x)=\left(\sqrt{ \lambda_1} \phi_{1}(x), \sqrt{ \lambda_2} \phi_{2}(x), \ldots, \sqrt{\lambda_m} \phi_{m}(x)\right)\]
for all $x\in X$,
 with kernel $K = K_B$ and the associated operator $T_K = T_{K_B}$ (with eigensystem ($\lambda_k, \phi_k$)).
Evaluating $f_m(x)$ at $x_i \in X_n$, we obtain the following finite embedding:
\begin{align*}
f_m(x_i) & =\left(\sqrt{\lambda_1} \phi_{1} (x_i), \sqrt{ \lambda_2} \phi_{2}(x_i), \ldots, \sqrt{ \lambda_m} \phi_{m}(x_i)\right)\\
& = \left((\lim _{n \to \infty} \sqrt{ \lambda_{1,n}} \cdot \phi_{1,n})(x_i), (\lim _{n \to \infty} \sqrt{\lambda_{2,n}}  \cdot \phi_{2,n})(x_i), \ldots, (\lim _{n \to \infty}\sqrt{ \lambda_{m,n}}  \cdot \phi_{m,n})(x_i)\right)\\
& = \left(\left(\lim _{n \to \infty} \sqrt{\frac{\ell_1}{n}} \cdot \sqrt n v^{(i)}_1 \right), \left(\lim _{n \to \infty} \sqrt{ \frac{\ell_2}{n}}  \cdot \sqrt n v^{(i)}_2 \right), \ldots, \left(\lim _{n \to \infty} \sqrt{ \frac{\ell_m}{n}}  \cdot \sqrt n v^{(i)}_m  \right) \right)\\
& = \left(\lim_{n \to \infty}(\sqrt{ \ell_1} v^{(i)}_1),\lim_{n \to \infty}(\sqrt{ \ell_2} v^{(i)}_2), \ldots, \lim_{n \to \infty}(\sqrt{ \ell_m} v^{(i)}_m) \right),
\end{align*}
which is the finite MDS embedding of $X_n$ into $\R^m$, where the eigensystem of the $n \times n$ inner-product matrix $B$ is denoted by $(\ell_k(n), v_k(n))$. 
The second equality above is from Theorem~\ref{thm: bengioConvergence}, and the third equality above is from Lemma~\ref{lem: evals-1}.

\section{Convergence of MDS for Finite Measures}\label{sec: convergence deterministic}

Though we gave no proofs in the above section, we do so now in the simpler deterministic case when points are not drawn from $X$ at random, but instead we assume that $\mu_n=\frac{1}{n}\sum_{i=1}^n\delta_{x_i}$ is sequence of measures with the support of $\mu_n$ consisting of $n$ points $\{x_1,\ldots, x_n\}\subseteq X$, and we assume that $\mu_n$ converges to some underlying probability distribution $\mu$ on $X$.
Our reason for working deterministically instead of randomly here is so that in Section~\ref{sec: convergence deterministic arbitrary}, we may consider the convergence of MDS in the more general setting when $\mu_n\to\mu$ are arbitrary probability measures; for example each $\mu_n$ may have infinite support. 

Figure~\ref{fig:sub:measure_Benjio} illustrates the case when the points are sampled i.i.d as discussed Section~\ref{sec: convergence random}, in contrast to Figure~\ref{fig:sub:measure_deterministic} which illustrates the case when the points are sampled in a manner that guarantees convergence of the measures $\mu_n$ to $\mu$.

\begin{figure}[h]

\subfloat[Points sampled i.i.d..]{\label{fig:sub:measure_Benjio}\includegraphics[width=0.35\textwidth]{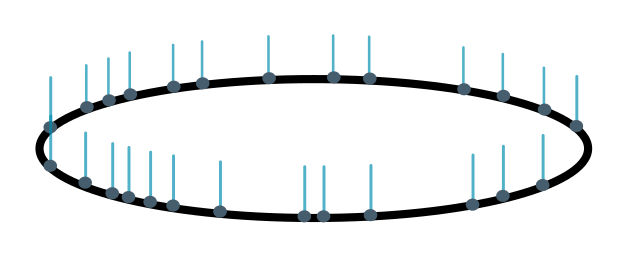}}
 \subfloat[Points sampled deterministically.]{\label{fig:sub:measure_deterministic}\includegraphics[width=0.35\textwidth]{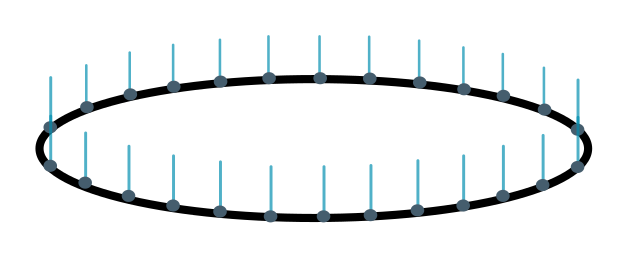}}

\caption{Illustration of the different notions of convergence of measures.}
\label{fig: deterministic_finite}

\end{figure}

\subsection{Preliminaries}

We begin by giving some background on kernels, convergence in total variation, and H\"{o}lder's Inequality.

\begin{definition}\label{def:metric-to-kernel}
Given a bounded metric measure space $(X, d,\mu)$, its \emph{associated kernel} $K\colon X\times X\to \R$ is 
\[ K(x,y)= \frac{1}{2} \left(-d(x,y)^2 + \int \limits_X  d(w,y)^2 \mu(\diff w) + \int \limits_X d(x,z)^2 \mu(\diff z) - \int_{X \times X} d(w,z)^2 \mu(\diff w \times \diff z) \right),\]
and its \emph{associated linear operator} $T_K\colon L^2(X, \mu) \to L^2(X, \mu)$ is defined via
\[[T_K\phi](x) = \int \limits_X K(x,y) \phi(y) \mu(\diff y).\]
\end{definition}

\begin{definition} [Total-variation convergence of measures]
Let $(X,{\mathcal {F}}$) be a measurable space. The total variation distance between two (positive) measures $\mu$ and $\nu$ is then given by

\[ \left\|\mu -\nu \right\|_{\text{TV}}=\sup _{f}\left\{\int _{X}f\,d\mu -\int _{X}f\,d\nu \right\}.\]  

The supremum is taken over $f$ ranging over the set of all measurable functions from $X$ to $[-1, 1]$. In our definition of metric measure spaces, we consider Borel probability measure on the space $X$.
\end{definition} 
Indeed, convergence of measures in total-variation implies convergence of integrals against bounded measurable functions, and the convergence is uniform over all functions bounded by any fixed constant.

\begin{theorem}[H\"{o}lder's Inequality]\label{thm: Holder}
Let $(S, \Sigma, \mu)$ be a measure space and let $p, q \in [1, \infty]$ with $\frac{1}{p} + \frac{1}{q} = 1$. Then, for all measurable real- or complex-valued functions $f$ and $g$ on S,
\[\|fg\|_{1}\leq \|f\|_{p}\|g\|_{q}.\]
\end{theorem}

\begin{lemma}\label{lem: L1L2}
For a measure space $X$ with finite measure ($\mu (X) < \infty)$, $L^2(X)$ is contained in $L^1(X)$.
\end{lemma}

\begin{proof}
Using the Schwarz' inequality (H\"{o}lder's inequality for $p,q =2$), we have the following:
\[\int\limits_X | f(x)| \mu (\diff x) = \int\limits_X 1 \cdot| f(x)| \mu (\diff x) \leq \|1\|_2 \| f\|_2 < \infty, \]
since $\int\limits_X 1 \mu (\diff x) = \mu(X) < \infty$. Thus, $L^2(X) \subseteq L^1(X).$
\end{proof}

\subsection{Convergence of MDS for Finite Measures}

Let $X$ be a bounded metric space, and let $\mu_n=\frac{1}{n}\sum_{i=1}^n\delta_{x_i}$ for $n\ge 1$ be a sequence of averages of Dirac deltas with $\{x_1,\ldots,x_n\}\subseteq X$, and with $\mu_n$ converging to $\mu$ in total variation as $n \to \infty$.
Each $\mu_n$ has finite support whereas $\mu$ may have infinite support.
Let $K_n$ and $T_{K_n}$ be the kernel and operator associated to $(X, d, \mu_n)$, and let $K$ and $T_K$ be the kernel and operator associated to $(X,d,\mu)$.
Let $\lambda_{k,n}$ denote the $k$th eigenvalue of $T_{K_n}$ with associated eigenfunction $\phi_{k,n}$. 
Furthermore, let $\lambda_k$ denote the $k$th eigenvalue of $T_K$ with associated eigenfunction $f_{k}$.

Following the proofs of~\cite{bengio2004learning,koltchinskii2000random}, we will show that the MDS embeddings of the metric measure space $(X, d, \mu_n)$ converge to the MDS embedding of $(X, d, \mu)$ in this deterministic setting.

By definition of $K$ and $K_n$, they are both bounded since $d$ is a bounded $L^2$-kernel with respect to the measures $\mu$ and $\mu_n$ for all $n$.
It follows from the total variation convergence of $\mu_n \to \mu$ that $\|K_n - K\|_\infty\to 0$. 

We use the result from~\cite[Theorem~3.1]{koltchinskii2000random}, where instead of studying finite spaces $X_n$, we study the possibly infinite space $X$ equipped with a measure $\mu_n$ of finite support. Since $\mu_n$ has finite support, Theorem~\ref{thm:koltchinskii} still holds under the assumptions of this section.

Suppose for each $k$, the eigenfunction $\phi_{k,n}$ of $T_{K_n}$ converges uniformly to some function $\phi_{k,\infty}$ as $n \to \infty$. In Proposition~\ref{prop: wasserstein convergence finite}, we show that $\phi_{k,\infty}$ are the eigenfunctions of $T_K$.

\begin{lemma}\label{lem: bounded eigenfunctions}
$\phi_{k,\infty}$ is bounded.
\end{lemma}

\begin{proof}
We first show that $\phi_{k,n}(x)$ is bounded. Indeed,
\begin{align*}
\left | \phi_{k,n}(x) \right |= \left |\frac{1}{\lambda_{k,n}}\int\limits_X K_n(x,y) \phi_{k,n}(y)  \mu_n(\diff y) \right| & \leq  \frac{1}{\left | \lambda_{k,n} \right |}  \int\limits_X \left | K_n(x,y) \right | \left | \phi_{k,n}(y) \right | \mu_n(\diff y) \\
& \leq  \frac{c}{\left | \lambda_{k,n} \right |}  \int\limits_X \left | \phi_{k,n}(y) \right |  \mu_n(\diff y). \\
\end{align*}
The second inequality follows from the fact that $K_n$ is bounded by some constant $c$. Furthermore, $\phi_{k,n}(x) \in L^2(X)$ and $\mu_n (X) = 1 < \infty$. 
It follows from Lemma~\ref{lem: L1L2} that $\phi_{k,n}(x) \in L^1(X)$, i.e.\ that $ \int\limits_X \left | \phi_{k,n}(y) \right |  \mu_n(\diff y) < \infty$.
Furthermore, knowing that $\|\phi_{k,n}-\phi_{k,\infty}\|_\infty\to 0$ and $\ell^2(\lambda(T_{ K_n}), \lambda(T_K)) \to 0$,  we deduce that $\phi_{k,\infty}$ is bounded.
\end{proof}

\begin{proposition}\label{prop: wasserstein convergence finite}
Suppose $\mu_n=\frac{1}{n}\sum_{x\in X_n}\delta_x$ converges to $\mu$ in total variation.
If the eigenfunctions $\phi_{k,n}$ of $T_{K_n}$ converge uniformly to $\phi _{k, \infty}$ as $n \to \infty$, then their limit are the corresponding eigenfunctions of $T_K$.
\end{proposition}

\begin{proof}
We have the following,
\begin{align*}
\phi_{k,n}(x)=&\frac{1}{\lambda_{k,n}}\int\limits_X K_n(x,y) \phi_{k,n}(y)  \mu_n(\diff y)\\
=&\frac{1}{\lambda_k}\int\limits_X K(x,y) \phi_{k,\infty}(y) \mu_n(\diff y)\\
&+\frac{\lambda_k-\lambda_{k,n}}{\lambda_{k,n}\lambda_k}\int\limits_X K(x,y) \phi_{k,\infty}(y) \mu_n(\diff y)\\
&+\frac{1}{\lambda_{k,n}}\int\limits_X  \Bigl(K_n(x,y)-K(x,y)\Bigr)\phi_{k,\infty}(y) \mu_n(\diff y)\\
&+\frac{1}{\lambda_{k,n}}\int\limits_X K_n(x,y) \Bigl(\phi_{k,n}(y)-\phi_{k,\infty}(y)\Bigr)\mu_n(\diff y).
\end{align*}

By Lemma~\ref{lem: bounded eigenfunctions},
$\phi_{k,\infty}$ is bounded. Therefore, we can insert $\frac{1}{\lambda_k}\int\limits_X K(x,y)\phi_{k,\infty}(y) \mu(\diff y)$ into the above aligned equations in order to obtain,
\begin{align*}
&\left|\phi_{k,n}(x)-\frac{1}{\lambda_k}\int\limits_X K(x,y)\phi_{k,\infty}(y)\mu(\diff y)\right|\\
\le&\left|\frac{1}{\lambda_k}\int\limits_X K(x,y) \phi_{k,\infty}(y) \mu_n(\diff y)-\frac{1}{\lambda_k}\int\limits_X K(x,y)\phi_{k,\infty}(y) \mu(\diff y)\right|\\
&+\left|\frac{\lambda_k-\lambda_{k,n}}{\lambda_{k,n}\lambda_k}\int\limits_X K(x,y) \phi_{k,\infty}(y) \mu_n(\diff y)\right|\\
&+\left|\frac{1}{\lambda_{k,n}}\int\limits_X  \Bigl(K_n(x,y)-K(x,y)\Bigr)\phi_{k,\infty}(y) \mu_n(\diff y)\right|\\
&+\left|\frac{1}{\lambda_{k,n}}\int\limits_X K_n(x,y) \Bigl(\phi_{k,n}(y)-\phi_{k,\infty}(y)\Bigr) \mu_n(\diff y)\right|\\
:=&A_n+B_n+C_n+D_n.
\end{align*}

Since the $\lambda_{k,n}$ converge to $\lambda_k$, since the $K_n$ converge to $K$, since the $\phi_{k,n}$ converge to $\phi_{k,\infty}$, and since $\phi_{k,\infty}$, $K$, and $K_n$ are bounded, it follows that the $B_n$, $C_n$, and $D_n$ converge to 0 as $n\to\infty$.
Since $\mu_n\to\mu$, we also have that $A_n\to0$ as $n \to\infty$
Therefore
\[\phi_{k,n}(x)\to \frac{1}{\lambda_k}\int\limits_X K(x,y)\phi_{k,\infty}(y)\mu(\diff y)=\frac{1}{\lambda_k}[T_k \phi_{k,\infty}](x)\]
for all $x\in X$.
Since we also have $\phi_{k,n}(x)\to \phi_{k,\infty}(x)$, it follows that $\lambda_k \phi_{k,\infty}(x)=T_k \phi_{k,\infty}$.
Therefore $\phi_{k,\infty}$ is an eigenfunction of $T_K$ with an eigenvalue $\lambda_k$.
\end{proof}

\section{Convergence of MDS for Arbitrary Measures}\label{sec: convergence deterministic arbitrary}

We now generalize the setting of Section~\ref{sec: convergence deterministic} to allow for arbitrary measures, as opposed to finite sums of Dirac delta measures as illustrated in Figure~\ref{fig:sub:measure_deterministic_finite}.
Suppose $X$ is a bounded metric space, and $\mu_n$ is an arbitrary sequence of probability measures on $X_n$ for all $n\in \N$, such that $\mu_n$ converges to $\mu$ in total variation as $n\to\infty$. 
For example, the support of each $\mu_n$ is now allowed to be infinite as illustrated in Figure~\ref{fig:sub:measure_deterministic_infinite}.
We will give some first results towards showing that the MDS embeddings of $(X, d, \mu_n)$ converge to the MDS embedding of $(X, d, \mu)$.

    \begin{figure}[h]
    
    \subfloat[Convergence of arbitrary measures with finite support.]{\label{fig:sub:measure_deterministic_finite}\includegraphics[width=0.8\textwidth]{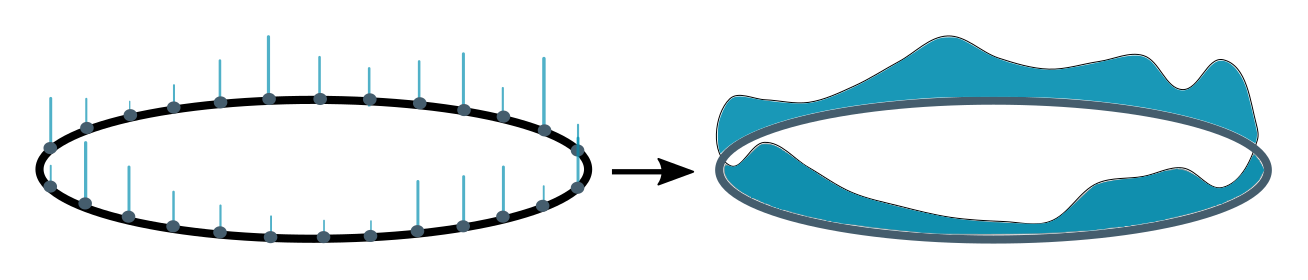}}
    
    \subfloat[Convergence of arbitrary measures with infinite support.]{\label{fig:sub:measure_deterministic_infinite}\includegraphics[width=0.8\textwidth]{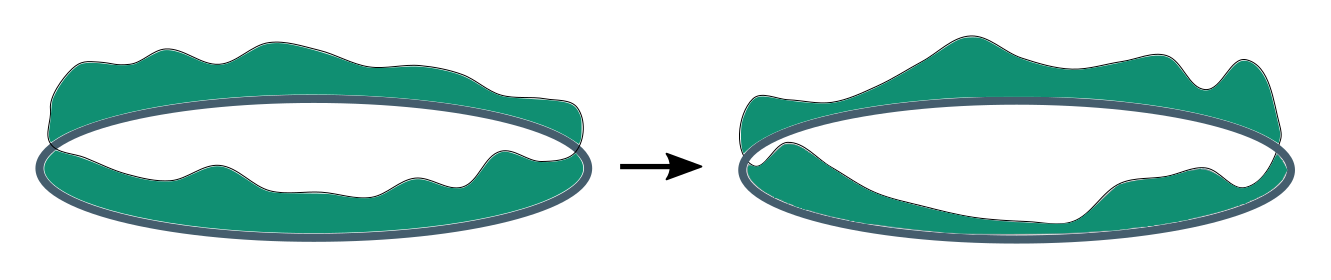}}
        \caption{Illustration of convergence (in total variation) of arbitrary measures.}
        
    \label{fig: arbitrary measures} 

\end{figure}

The bounded metric measure space $(X, d, \mu)$ is equipped with a kernel $K\colon X\times X\to \R$ and linear operator $T_K\colon L^2(X, \mu) \to L^2(X, \mu)$, as defined as in Definition~\ref{def:metric-to-kernel}.
For $(X, d, \mu_n)$, we denote the analogous kernel by $K_n\colon X\times X\to \R$ and its linear operator by $T_{K_n}\colon L^2(X, \mu_n) \to L^2(X, \mu_n)$.
Let $\lambda_{k,n}$ denote the $k$th eigenvalue of $T_{K_n}$ with associated eigenfunction $\phi_{k,n}$. 
Furthermore, let $\lambda_k$ denote the $k$th eigenvalue of $T_K$ with associated eigenfunction $\phi_{k}$.

\begin{proposition}\label{prop: wasserstein convergence}
Suppose $\mu_n$ converges to $\mu$ in total variation.
If the eigenvalues $\lambda_{k,n}$ of $T_{K_n}$ converge to $\lambda_k$, and if their corresponding eigenfunctions $\phi_{k,n}$ of $T_{K_n}$ converge uniformly to $\phi_{k, \infty}$ as $n \to \infty$, then the $\phi_{k, \infty}$ are eigenfunctions of $T_K$ with eigenvalue $\lambda_k$.
\end{proposition}

\begin{proof}
The same proof of Proposition~\ref{prop: wasserstein convergence finite} holds.
Indeed, so long as we know that the eigenvalues $\lambda_{k,n}$ of $T_{K_n}$ converge to $\lambda_k$, then nowhere else in the proof of Proposition~\ref{prop: wasserstein convergence finite} does it matter whether $\mu_n$ is an average of Dirac delta masses or instead an arbitrary probability measure.
\end{proof}

We conjecture that the hypothesis in Proposition~\ref{prop: wasserstein convergence} about the convergence of eigenvalues is unnecessary.

\begin{conjecture}\label{conj:eigenvalue-convergence}
Suppose we have the convergence of measures $\mu_n\to\mu$ in total variation.
The ordered spectrum of $T_{K_n}$ converges to the ordered spectrum of $T_{K}$ as $n \to \infty$ with respect to the $\ell^2$--distance,
\[\ell^2(\lambda(T_{ K_n}), \lambda(T_K)) \to 0.\]
\end{conjecture}

\begin{remark}\label{rem:eigenvalue-convergence}
We remark that some ideas from the proof of~\cite[Theorem~3.1]{koltchinskii2000random} may be useful here.
One change is that the inner products considered in equation (3.4) of the proof of~\cite[Theorem~3.1]{koltchinskii2000random} may need to be changed to inner products with respect to the measure $\mu_n$.
\end{remark}

\section{Convergence of MDS with Respect to Gromov--Wasserstein Distance}\label{sec: convergence GW}

We now consider the more general setting in which $(X_n, d_n, \mu_n)$ is an arbitrary sequence of metric measure spaces, converging to $(X, d, \mu)$ in the Gromov--Wasserstein distance as illustrated in Figure~\ref{fig:sub:measure_gw_finite} for the finite case and Figure~\ref{fig:sub:measure_gw_infinite} for the infinite case.
We remark that $X_n$ need to no longer equal $X$, nor even be a subset of $X$. Indeed, the metric $d_n$ on $X_n$ is allowed to be different from the metric $d$ on $X$.
Sections~\ref{sec: convergence random} and~\ref{sec: convergence deterministic} would be the particular case (depending on your perspective) when either $X_n$ is a finite subset of $X$ and $d_n$ is the restriction of $d$, or equivalently when $(X_n,d_n)$ is equal to $(X,d)$ but $\mu_n$ is a finite average of Dirac delta masses.
Section~\ref{sec: convergence deterministic arbitrary} is the particular case when $(X_n,d_n)=(X,d)$ for all $n$, and the measures $\mu_n$ are converging to $\mu$.
We now want to consider the case where metric $d_n$ need no longer be equal to $d$.

\begin{figure}[h] 

    \subfloat[Convergence of mm-spaces equipped with measures of finite support.]{\label{fig:sub:measure_gw_finite}\includegraphics[width=0.8\textwidth]{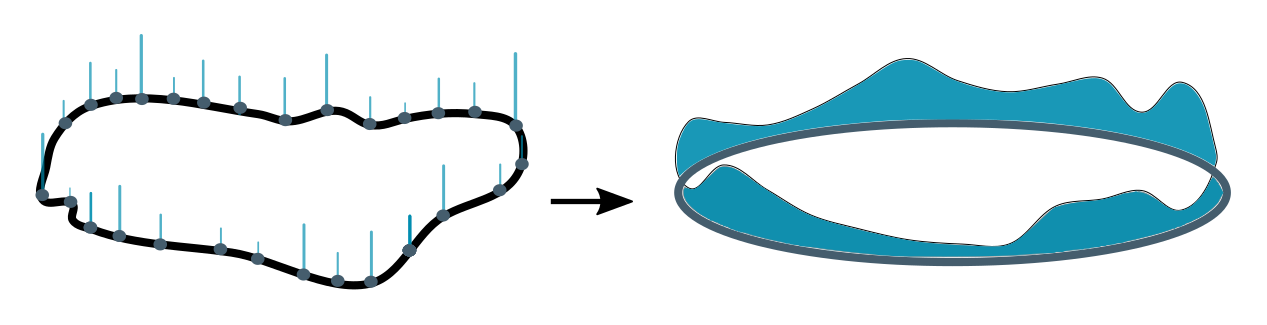}}
    
    \subfloat[Convergence of mm-spaces equipped with measures of infinite support.]{\label{fig:sub:measure_gw_infinite}\includegraphics[width=0.8\textwidth]{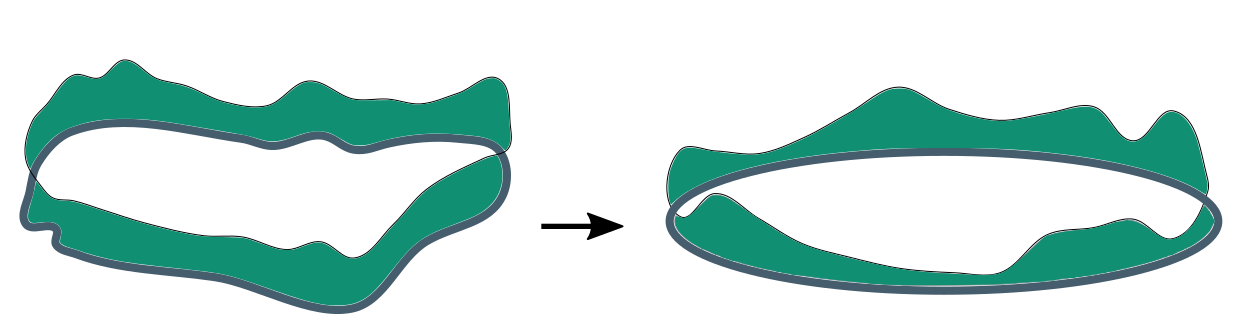}}

    \caption{Illustration of Gromov--Wasserstein convergence of arbitrary metric measure spaces (mm-spaces).}

    \label{fig: GW convergence} 

\end{figure}

\begin{conjecture}\label{conj: GW}
Let $(X_n, d_n, \mu_n)$ for $n\in\N$ be a sequence of metric measure spaces that converges to $(X, d, \mu)$ in the Gromov--Wasserstein distance. Then the MDS embeddings converge.
\end{conjecture}

\begin{question}\label{ques: GW}
Are there other notions of convergence of a sequence of arbitrary (possibly infinite) metric measure spaces $(X_n, d_n, \mu_n)$ to a limiting metric measure space $(X, d, \mu)$ that would imply that the MDS embeddings converge in some sense?
We remark that one might naturally try to break this into two steps: first analyze which notions of convergence $(X_n, d_n, \mu_n) \to (X, d, \mu)$ imply that the operators $T_{K_n}\to T_K$ converge, and then analyze which notions of convergence on the operators $T_{K_n}\to T_K$ imply that their eigendecompositions and MDS embeddings converge. 
\end{question}

\chapter{Conclusion}
\label{chap: conclusion}

MDS is concerned with problem of mapping the objects $x_1, \ldots, x_n$ to a configuration (or embedding) of points $f(x_1), \ldots, f(x_n)$ in $\R^m$ in such a way that the given dissimilarities $d_{ij}$ are well-approximated by the Euclidean distances between $f(x_i) $ and $f(x_j)$. We study a notion of MDS on infinite metric measure spaces, which can be simply thought of as spaces of (possibly infinitely many) points equipped with some probability measure. 
We explain how MDS generalizes to infinite metric measure spaces. 
Furthermore, we describe in a self-contained fashion an infinite analogue to the classical MDS algorithm.
Indeed, classical multidimensional scaling can be described either as a $\strain$-minimization problem, or as a linear algebra algorithm involving eigenvalues and eigenvectors.
We describe how to generalize both of these formulations to infinite metric measure spaces.
We show that this infinite analogue minimizes a $\strain$ function similar to the $\strain$ function of classical MDS. This theorem generalizes~\cite[Theorem~14.4.2]{bibby1979multivariate}, or equivalently~\cite[Theorem~2]{trosset1997computing}, to the infinite case.
Our proof is organized analogously to the argument in~\cite[Theorem~2]{trosset1997computing}.

As a motivating example for convergence of MDS, we consider the MDS embeddings of the circle equipped with the (non-Euclidean) geodesic metric.
By using the known eigendecomposition of circulant matrices, we identify the MDS embeddings of evenly-spaced points from the geodesic circle into $\R^m$, for all $m$.
Indeed, the MDS embeddings of the geodesic circle are closely related to~\cite{von1941fourier}, which was written prior to the invention of MDS.

Lastly, we address convergence questions for MDS.
Indeed, convergence is well-understood when each metric space has the same finite number of points~\cite{sibson1979studies}, but we are also interested in convergence when the number of points varies and is possibly infinite.
We survey Sibson's perturbation analysis~\cite{sibson1979studies} for MDS on a fixed number of $n$ points. 
We survey results of~\cite{bengio2004learning,koltchinskii2000random} on the convergence of MDS when $n$ points $\{x_1,\ldots,x_n\}$ are sampled from a metric space according to a probability measure $\mu$, in the limit as $n\to\infty$.
We reprove these results under the (simpler) deterministic setting when points are not randomly chosen, and instead we assume that the corresponding finite measures $\mu_n = \frac{1}{n}\sum\limits_{i=1}^{n} \delta_{x_i}$ (determined by $n$ points) converge to $\mu$.
This allows us, to consider the more general setting where we have convergence of \emph{arbitrary} probability measures $\mu_n\to\mu$.
However, several questions remain open.
In particular, we would like to have a better understanding of the convergence of MDS under the most unrestrictive assumptions of a sequence of arbitrary (possibly infinite) metric measure spaces converging to a fixed metric measure space, perhaps in the Gromov--Wasserstein distance (that allows for distortion of both the metric and the measure simultaneously); see Conjecture~
\ref{conj: GW} and Question~\ref{ques: GW}.

Despite all of the work that has been done on MDS by a wide variety of authors, many interesting questions remain open (at least to us).
For example, consider the MDS embeddings of the $n$-sphere for $n\ge 2$.

\begin{question}
What are the MDS embeddings of the $n$-sphere $S^n$, equipped with the geodesic metric, into Euclidean space $\R^m$?
\end{question}

To our knowledge, the MDS embeddings of $S^n$ into $\R^m$ are not understood for all positive integers $m$ except in the case of the circle, when $n=1$.
The above question is also interesting, even in the case of the circle, when the $n$-sphere is not equipped with the uniform measure.
As a specific case, what is the MDS embedding of $S^1$ into $\R^m$ when the measure is not uniform on all of $S^1$, but instead (for example) uniform with mass $\frac{2}{3}$ on the northern hemisphere, and uniform with mass $\frac{1}{3}$ on the southern hemisphere?

We note the work of Blumstein and Kvinge \cite{blumstein2018letting}, where a finite group representation theoretic perspective on MDS is employed. Adapting these techniques to the analytical setting of compact Lie groups may prove fruitful for the case of infinite MDS on higher dimensional spheres. 

We also note the work \cite{blumstein2019pseudo}, where the theory of an MDS embedding into pseudo Euclidean space is developed.  In this setting, both positive and negative eigenvalues are used to create an embedding. In the example of embedding $S^1$, positive and negative eigenvalues occur in a one-to-one fashion. We wonder about the significance of the full spectrum of eigenvalues for the higher dimensional spheres.

\backmatter
\bibliographystyle{plain}

\bibliography{KassabMasters.bib}

\end{document}